\documentclass[a4paper,10pt, draft]{amsart}

\usepackage{amsmath}
\usepackage{amsthm}
\usepackage{amsfonts}
\usepackage{amssymb}
\usepackage{euscript}
\usepackage[all,matrix,arrow]{xy}

\newcommand{\PP}{\mathbb{P}}

\newcommand{\nc}{\newcommand}

\numberwithin{equation}{section}
\newtheorem{thm}{Theorem}[section]
\newtheorem{prop}[thm]{Proposition}
\newtheorem{lem}[thm]{Lemma}
\newtheorem{cor}[thm]{Corollary}
\theoremstyle{remark}
\newtheorem{rem}[thm]{Remark}

\newtheorem{example}[thm]{Example}
\newtheorem{dfn}[thm]{Definition}

\nc{\gl}{\mathfrak{gl}}
\nc{\GL}{\mathfrak{GL}}
\nc{\g}{\mathfrak{g}}
\nc{\gh}{\widehat\g}
\nc{\h}{\mathfrak{h}}
\nc{\la}{\lambda}
\nc{\al}{\alpha }
\nc{\ve}{\varepsilon }
\nc{\om}{\omega }

\nc{\ta}{\theta}
\nc{\veps}{\varepsilon}
\nc{\ch}{{\mathop {\rm ch}}}
\nc{\Tr}{{\mathop {\rm Tr}\,}}
\nc{\Id}{{\mathop {\rm Id}}}
\nc{\ad}{{\mathop {\rm ad}}}
\nc{\bra}{\langle}
\nc{\ket}{\rangle}
\nc{\x}{{\bf x}}
\nc{\bs}{{\bf s}}
\nc{\bp}{{\bf p}}
\nc{\bc}{{\bf c}}
\nc{\pa}{\partial}
\nc{\ld}{\ldots}
\nc{\cd}{\cdots}
\nc{\hk}{\hookrightarrow}
\nc{\T}{\otimes}
\newcommand{\bea}{\begin{equation}}
\newcommand{\ena}{\end{equation}}
\nc{\gr}{\mathrm{gr}}
\nc{\ov}{\overline}

\nc{\cO}{\mathcal O}
\nc{\cF}{\mathcal F}
\nc{\cL}{\mathcal L}
\nc{\msl}{\mathfrak{sl}}
\nc{\mgl}{\mathfrak{gl}}
\nc{\U}{\mathrm U}
\nc{\V}{\EuScript V}
\nc{\bH}{\EuScript H}
\nc{\Res}{\mathrm{Res\ }}

\newcommand{\bC}{{\mathbb C}}

\newcommand{\bP}{{\mathbb P}}
\newcommand{\bG}{{\mathbb G}}
\nc{\fA}{\mathfrak{A}}

\newcommand{\fa}{{\mathfrak a}}

\newcommand{\fb}{{\mathfrak b}}

\newcommand{\ft}{{\mathfrak t}}

\newcommand{\fn}{{\mathfrak n}}

\newcommand{\Fl}{\EuScript{F}}

\newcommand{\bS}{{\bf S}}

\newcommand{\bd}{{\bf d}}

\newcommand{\be}{{\bf e}}
\newcommand{\dimv}{{\bf dim}}

\begin{document}
\title[Degenerate flag varieties]
{Degenerate flag varieties: moment graphs and Schr\"oder numbers}
\author{Giovanni Cerulli Irelli, Evgeny Feigin, Markus Reineke}
\address{Giovanni Cerulli Irelli:\newline
Mathematisches Institut, Universit\"at Bonn, Bonn, Germany 53115}
\email{cerulli.math@googlemail.com}
\address{Evgeny Feigin:\newline
Department of Mathematics,\newline National Research University Higher School of Economics,\newline
Russia, 117312, Moscow, Vavilova str. 7\newline
{\it and }\newline
Tamm Department of Theoretical Physics,
Lebedev Physics Institute
}
\email{evgfeig@gmail.com}
\address{Markus Reineke:\newline 
Fachbereich C - Mathematik, Bergische Universit\"at Wuppertal, D - 42097 Wuppertal}
\email{reineke@math.uni-wuppertal.de}

\begin{abstract}
We study geometric and combinatorial properties of the degenerate flag varieties of type A. 
These varieties are acted upon by the automorphism group of a certain
representation of a type A quiver, containing a maximal torus T.
Using the group action, we  describe the moment graphs, encoding the zero- and 
one-dimensional T-orbits.
We also study the smooth and singular loci of the degenerate flag varieties. 
We show that the 
Euler characteristic of the smooth locus is equal to the large Schr\"oder number and the Poincar\'e polynomial
is given by a natural statistics counting the number of diagonal steps in a Schr\"oder path. 
As an application we obtain a new combinatorial description of the large and small Schr\"oder numbers 
and their q-analogues. 
\end{abstract}

\maketitle

\section{Introduction}
For $n\ge 1$, let $\Fl^a_{n+1}$ be the degenerate flag variety attached to the Lie algebra $\msl_{n+1}$ (see
\cite{Fe1}, \cite{Fe2}). This is a flat degeneration of the classical flag variety, defined
using the PBW filtration on irreducible representations of $\msl_{n+1}$ (see \cite{FFoL}). 
By construction, the 
$\Fl^a_{n+1}$ is acted upon by the degenerate Lie group $SL_{n+1}^a$, which is the semi-direct
product of the Borel subgroup $B$ and the abelian group $\bG_a^N$, where $\bG_a$ is the additive
group of the field. In particular, $\bG_a^N$ acts on $\Fl^a_{n+1}$ with an open dense orbit.
The degenerate flag varieties are singular normal projective algebraic varieties, sharing many nice
properties with their classical analogues. In particular, they enjoy a description in terms of linear algebra
as subvarieties inside a product of Grassmann varieties.

It has been observed in \cite{CFR} that the degenerate flag varieties can be identified with certain
quiver Grassmannians of the equioriented quiver of type $A_n$. More precisely, 
$\Fl^a_{n+1}$ is isomorphic to the quiver Grassmannian ${\rm Gr}_{\dimv A}(A\oplus A^*)$, where
$A$ and $A^*$ are
the path algebra of the equioriented $A_n$ quiver, resp.~its dual. This observation was used in two different ways:
first, to get a deeper understanding of the geometry and combinatorics of the degenerate flag varieties, and,
second, to generalize the results and constructions to a wider class of quiver Grassmannians.
In this paper we continue the study of the varieties $\Fl^a_{n+1}$ using the techniques from the theory
of quiver Grassmannians. More concretely, we achieve two things: first, we describe the combinatorial structure
of the moment graph of $\Fl^a_{n+1}$. Second, we describe explicitly the smooth and singular loci 
of the degenerate flag varieties. Let us give a brief description of our results.   

Recall that the notion of the moment graph attached to an algebraic variety $X$ acted upon by an algebraic
torus was introduced in \cite{GKM}, \cite{BM}. This combinatorial object captures the structure of zero- and 
one-dimensional orbits of $T$. It turns out to be very useful for describing various geometric properties
of $X$, such as cohomology and intersection cohomology. Our first task is to describe the moment graph $\Gamma$ of 
$\Fl^a_{n+1}$. We note that the automorphism group ${\rm Aut}(A\oplus A^*)$ acts on $\Fl^a_{n+1}$.
The maximal torus $T$ of  the automorphism group acts with a finite number of fixed points (this
number is equal to the normalized median Genocchi number, see \cite{CFR},\cite{Fe2},\cite{Fe3}). 
It is proved in \cite{CFR} that there exists a codimension one subgroup 
${\mathfrak A}\subset {\rm Aut}(A\oplus A^*)$ containing the torus $T$ such that ${\mathfrak A}$-orbits through $T$-fixed points 
are affine cells that provide a cellular decomposition of $\Fl^a_{n+1}$. 
We describe ${\mathfrak A}$ as a quotient of the Borel subgroup of $SL_{2n}$. Using this description, 
we prove the following theorem (for a more precise formulation see section \ref{MG}):
\begin{thm}
The number of one-dimensional $T$-orbits in $\Fl^a_{n+1}$ is finite. 
The edges of $\Gamma$ correspond to the one-parameter subgroups of ${\mathfrak A}$.
\end{thm}
We note that the structure of $\Gamma$ has many common features with its classical analogue (see 
\cite{C}, \cite{GHZ}, \cite{T}).

Our next goal is to describe the smooth locus of the degenerate flag varieties. Since $\Fl^a_{n+1}$
has a cellular decomposition by ${\mathfrak A}$-orbits of $T$-fixed points, it suffices to decide
which $T$-fixed
points are smooth. We recall that the $T$-fixed points are labeled by collections $\bS=(S_1,\dots,S_n)$
of subsets of $\{1,\dots,n+1\}$ such that $\#S_i=i$ and $S_i\subset S_{i+1}\cup\{i+1\}$. We denote
the corresponding $T$-fixed point by $p_{\bS}$.
\begin{thm}\label{I2}
A point $p_{\bS}$ is smooth if and only if for all $1\leq j<i\leq n$, the condition 
$i\in S_j$ implies $j+1\in S_i$. The number of smooth $T$-fixed points is given by the large 
Schr\"oder number $r_n$.
\end{thm}

We recall (see \cite{St}, \cite{G}) that the large Schr\"oder number $r_n$ is equal to the number 
of Schr\"oder paths, i.e. subdiagonal
lattice paths starting at $(0,0)$ and ending at $(n,n)$ with the following steps allowed: $(1,0)$,
$(0,1)$ and $(1,1)$.  
In particular, Theorem \ref{I2} implies that the Euler characteristic of the smooth locus of $\Fl^a_{n+1}$ is equal 
to $r_n$. Moreover we prove the following theorem:
\begin{thm}
The Poincar\'e polynomial of the smooth locus of $\Fl^a_{n+1}$ is equal to the (scaled) 
$q$-Schr\"oder number $q^{n(n-1)/2}r_n(q)$, where $r_n(q)$ is defined via the statistics on Schr\"oder
paths, counting the number of $(1,1)$ steps in a path. 
\end{thm} 

As an application, we obtain a new proof of the statement that $r_n(q)$ is divisible by $1+q$. The ratio
is known to give a $q$-analogue of the small Schr\"oder numbers.

Let us mention two more results of the paper. First, we prove that, for a general Dynkin type
quiver $Q$ and a projective $Q$-module $P$ and an injective $Q$-module $I$, the quiver Grassmannian
${\rm Gr}_{\dimv P} (P\oplus I)$ is smooth in codimension $2$. Second, we prove that the smooth locus
of $\Fl^a_{n+1}$ can be described as the subvariety of points where the desingularization map 
$R_{n+1}\to \Fl^a_{n+1}$ (see \cite{FF}) is one-to-one.

Finally, we note that all the results of the paper can be generalized to the case of the degenerate partial
(parabolic) flag varieties. 

Our paper is organized as follows:\\
In Section $1$ we introduce the main objects and recall the main definitions and results needed in the rest of 
the paper.\\
In Section $2$ we describe the moment graph of the degenerate flag varieties.\\
In Section $3$ we prove a criterion for smoothness of a $T$-fixed point and 
compute the Euler characteristics and Poincar\'e polynomials.\\
In Appendix $A$ we prove the regularity in codimension $2$ of certain quiver Grassmannians.\\
In Appendix $B$ we describe the smooth locus in terms of the desingularization.\\
In Appendix $C$ we compute the moment graph for the degenerate flag variety $\Fl^a_4$.\\

\section{Quiver Grassmannians and degenerate flag varieties}
In this section we recall definitions and results on the degenerate flag varieties and quiver 
Grassmannians to be used in the main body of the paper. 

\subsection{Degenerate flag varieties}
Let $\Fl_{n+1}$ be the complete flag variety for the group $SL_{n+1}$, i.e.~the quotient $SL_{n+1}/B$
by the Borel subgroup $B$. This variety has an explicit realization as the subvariety  of the product 
of Grassmannians $\prod_{k=1}^n {\rm Gr}_k(\bC^{n+1})$ consisting of collections $(V_1,\dots,V_n)$ such
that $V_i\subset V_{i+1}$ for all $i$. In \cite{Fe1},\cite{Fe2} flat degenerations $\Fl^a_{n+1}$ of the
classical flag varieties were introduced. The degenerate flag varieties $\Fl^a_{n+1}$ are (typically
singular) irreducible normal projective algebraic varieties, sharing many nice properties with their 
classical analogues. In particular, they also have a very explicit description in linear algebra terms.
Namely, let $W$ be an $(n+1)$-dimensional vector space  with a basis $w_1,\dots,w_{n+1}$. 
Let $pr_k:W\to W$ be the projection operators defined by
$pr_k w_k=0$ and $pr_kw_i=w_i$ if $i\ne k$. The following lemma is proved in \cite{Fe2}, Theorem $2.1$. 
\begin{lem}\label{Lemma:DegFlagGrass}
The degenerate flag variety $\Fl^a_{n+1}$ is a
subvariety of the product of Grassmannians $\prod_{k=1}^n {\rm Gr}_k(W)$,
consisting of collections $(V_k)_{k=1}^n$ such that 
$$pr_{k+1} V_k\subset V_{k+1}\ \text{ for all }\ k=1,\dots,n-1.$$  
\end{lem}

Another important property of the varieties $\Fl^a_{n+1}$ is that 
they admit a cellular decomposition into a disjoint union of complex cells. Moreover, there exists an algebraic
group ${\mathfrak A}$ and a torus $T\subset {\mathfrak A}$ acting on $\Fl^a_{n+1}$ such that each cell contains exactly one $T$-fixed
point and the ${\mathfrak A}$-orbit through this point coincides with the cell. Let us describe the combinatorics of  
the cells, postponing the description of the group action to the next subsection. So let
$\bS=(S_1,\dots,S_n)$ be a collection of subsets of the set $\{1,\dots,n+1\}$ such that each $S_i$ 
contains $i$
elements. Then the cells in $\Fl^a_{n+1}$ are labeled by the collections satisfying the following property 
\begin{equation}
\label{mG}
S_k\subset S_{k+1}\cup \{k+1\},\quad k=1,\dots,n-1.
\end{equation}   
We call such collections admissible.
The number of admissible collections (and hence the Euler characteristic of $\Fl^a_{n+1}$) is equal to the
normalized median Genocchi number $h_{n+1}$ (see \cite{Fe2},\cite{Fe3}, \cite{CFR}).
We note that the correspondence between the admissible collections and $T$-fixed points
is very explicit. Namely, for a collection $\bS$ we denote by $p_{\bS}\in \Fl^a_{n+1}$ a point
defined by
\[
p_{\bS}=(V_1,\dots,V_n),\qquad V_k=\mathrm{span}(w_i,\ i\in S_k).
\]
Clearly, such a point belongs to $\Fl^a_{n+1}$ if and only if the collection $\bS$ is admissible.

\subsection{Quiver Grassmannians.}\label{QGr} 
The construction above can be reformulated in the language of quiver
Grassmannians (see e.g. \cite{Sc}, \cite{CR}). Let $Q$ be  the equioriented type $A_n$ quiver 
with vertices labeled by numbers from $1$ 
to $n$ and arrows $i\to i+1$, $i=1,\dots,n-1$: 
\[
Q:\qquad \bullet \to \bullet \to \dots\to\bullet
\] 
For a representation $M$ of $Q$ we denote by $M_k$ the subspace of $M$ attached to the vertex $k$. 
For a pair $1\le i\le j\le n$ let $R_{i,j}$ be 
an indecomposable representation of $Q$ supported on the vertices $i,\dots,j$ (i.e. $(R_{i,j})_k=\bC$ for
$i\le k\le j$ and is trivial otherwise). 
We have the following immediate lemma.
\begin{lem}\label{HE}
\[
\dim {\rm Hom} (R_{i,j}, R_{k,l})=\begin{cases} 1, \text{ if } k\le i\le l\le j,\\ 
0, \text{ otherwise}\end{cases};
\]
\[
\dim {\rm Ext}^1 (R_{i,j}, R_{k,l})=\begin{cases} 1, \text{ if } i+1\leq k\leq j+1\leq l,\\ 
0, \text{ otherwise}\end{cases}.
\]
\end{lem}

We note that the representations $R_{1,j}$ are injective 
and the $R_{i,n}$ are projective (note that these are all indecomposable injective and projective 
representations of $Q$). We set 
$$I_k=R_{1,k},\quad  P_k=R_{k,n}, \quad P=\bigoplus_{k=1}^n P_k,\quad  I=\bigoplus_{k=1}^n I_k.$$ 
Hence, $P$
is isomorphic to the path algebra of $Q$ and $I$ is isomorphic to its linear dual. For a dimension vector 
$\be=(e_1,\dots,e_n)$ and a representation $M$ of $Q$, we denote by ${\rm Gr}_\be (M)$ the quiver Grassmannian
of $\be$-dimensional subrepresentations of $M$. Then by definition one gets
\begin{equation}
\label{QG}
\Fl^a_{n+1}\simeq {\rm Gr}_{\dimv P} (P\oplus I).
\end{equation}   

\begin{rem}\label{Rem:CoeffQuiver}
The representation $P\oplus I$ can be visualized by the following picture (here $n=4$). Each fat dot
corresponds to a basis vector and two dots corresponding to the vectors $u$ and $v$ are connected 
by an arrow $u\to v$ if $u$ is mapped to $v$. The quiver obtained in this way is called the coefficient--quiver of $P\oplus I$.
\begin{equation}\label{pic}
\xymatrix@R=6pt@C=8pt
{
\bullet\ar[r]&\bullet\ar[r]&\bullet\ar[r]&\bullet&\\
\bullet\ar[r]&\bullet\ar[r]&\bullet&\bullet&\\
\bullet\ar[r]&\bullet&\bullet\ar[r]&\bullet&\\
\bullet&\bullet\ar[r]&\bullet\ar[r]&\bullet&\\
\bullet\ar[r]&\bullet\ar[r]&\bullet\ar[r]&\bullet&\\
}
\end{equation}

\end{rem}
The isomorphism \eqref{QG} has many important consequences. In particular the automorphism group of 
the $Q$-module
$P\oplus I$ acts on $\Fl^a_{n+1}$. The group ${\rm Aut} (P\oplus I)$ is of the following form
$
{\rm Aut} (P\oplus I)=
\begin{pmatrix}
{\rm Aut} P & {\rm Hom} (I,P)\\
{\rm Hom} (P,I) & {\rm Aut} I
\end{pmatrix}.  
$ 
The part ${\rm Hom} (I,P)$ is one-dimensional (${\rm Hom} (I,P)={\rm Hom} (I_n,P_1)$). 
We denote by ${\mathfrak A}\subset {\rm Aut} (P\oplus I)$ the following subgroup 
$${\mathfrak A}=\begin{pmatrix}
{\rm Aut} P & 0\\
{\rm Hom} (P,I) & {\rm Aut} I
\end{pmatrix}.$$  
The group ${\mathfrak A}$ contains a torus $T$ isomorphic to $(\bC^*)^{2n}$, where each factor scales the corresponding
indecomposable summand in $P\oplus I$. The importance of the group ${\mathfrak A}$ comes from the following lemma, proved
in \cite{CFR}.
\begin{lem}
The group ${\mathfrak A}$ acts on $\Fl^a_{n+1}$ with a finite number of orbits. Each orbit is a complex affine cell,
containing exactly one $T$-fixed point. The orbits are labeled by admissible collections.
\end{lem}

For an admissible collection $\bS$ we denote by $C_{\bS}$ the cell containing the $T$-fixed point $p_{\bS}$.

\begin{rem}\label{torus}
We note that $T$ contains the identity automorphism which acts trivially on the degenerate flag variety.
Hence one gets a $(2n-1)$-dimensional torus acting effectively on $\Fl^a_{n+1}$, while the maximal
torus $T^c$ acting on the classical flag variety $\Fl_{n+1}$ is $n$-dimensional. We note that there 
is a natural embedding $T^c\subset T$. In fact recall that any point of $\Fl^a_{n+1}$ is of the form
$(V_k)_{k=1}^n$, $V_k\subset W\simeq \bC^{n+1}$. Hence any diagonal (in the basis $w_i$) matrix in $SL(W)$ 
induces an automorphism of the degenerate flag variety. Hence we obtain the embedding $T^c\subset T$.
\end{rem}

Finally, we note that the torus $T$ contains a one-dimensional subtorus $T_0$ with the following properties:
the set of $T$-fixed points coincides with the set of the $T_0$-fixed points and the attracting set of
a fixed point $p$ coincides the the orbit ${\mathfrak A}p$ (which is an affine cell) \cite[Theorem~5.1]{CFR}. The action of 
the one-dimensional torus can be illustrated as follows ($n=4$, the scalar $\lambda\in\bC^*$ is the parameter of the
torus and the power of $\la$  corresponds to the scaling factor of the $T_0$ action):
\begin{equation}\label{la}
\xymatrix@R=6pt@C=8pt
{
1&&&&&\bullet&\\
\la&&&&\bullet\ar[r]&\bullet&\\
\la^2&&&\bullet\ar[r]&\bullet\ar[r]&\bullet&\\
\la^3&&\bullet\ar[r]&\bullet\ar[r]&\bullet\ar[r]&\bullet&\\
\la^4&&\bullet\ar[r]&\bullet\ar[r]&\bullet\ar[r]&\bullet&\\
\la^5&&\bullet\ar[r]&\bullet\ar[r]&\bullet&&\\
\la^6&&\bullet\ar[r]&\bullet&&&\\
\la^7&&\bullet&&&&\\
}
\end{equation}
This picture is obtained from the picture \eqref{pic} by putting the $P$-part on top of the $I$-part.

We conclude this section by describing the action of the torus $T$ on the tangent space at a $T$--fixed point $p_\bS$. Recall that the tangent space at $p_\mathbf{S}$ is isomorphic to ${\rm Hom}(p_\mathbf{S},M/p_\mathbf{S})$ where $M=P\oplus I$ (\cite[Lemma~2.3]{CFR}, \cite{CR}, \cite{Sc}). Let $\theta_M$ be the coefficient quiver of M (see Remark \ref{Rem:CoeffQuiver}) and let 
$\pi: \theta_M\rightarrow Q$ be the natural projection onto the $A_n$ quiver $Q$. The coefficient 
quiver of $M/p_\mathbf{S}$ is $\theta_M\setminus \mathbf{S}$. The vector space 
${\rm Hom}(p_\mathbf{S},M/p_\mathbf{S})$ has a distinguished basis, denoted by $\mathcal{B}$,  
parameterized by triples $(A,f,B)$ where A  is a predecessor--closed  connected sub quiver of 
$\mathbf{S}$, $B$ is a successor--closed  connected sub quiver of $\theta_M\setminus\mathbf{S}$ 
and $f:A\rightarrow B$ is a quiver isomorphism compatible with $\pi$ (see \cite{CrawleyTree}). 
For example, in the left--hand side of the  picture below   
 \begin{equation}\label{Eq:ExTangSpaceBasis}
 \begin{array}{ccc}
 \xymatrix@R=6pt@C=8pt{
\cdot\ar[r]&\cdot\ar[r]&\bullet&\\
\cdot\ar[r]&\bullet&\cdot&\\
\bullet&*+[F]{\bullet}\ar[r]&\bullet &A\ar^f[d]\\
\cdot\ar[r]&*+[F]{\cdot}\ar[r]&\bullet& B
 }
 &&
 \xymatrix@R=6pt@C=8pt{
\lambda_3\ar[r]&\lambda_3\ar[r]&\lambda_3\\
\lambda_4\ar[r]&\lambda_4&1\\
\lambda_5&\lambda_1\ar[r]&\lambda_1\\
\lambda_2\ar[r]&\lambda_2\ar[r]&\lambda_2
 } 
 \end{array}
 \end{equation}
the fat dots highlight the coefficient--quiver $\mathbf{S}$ of a T--fixed point $p_\mathbf{S}$ of $\Fl_4^a$ and the frames highlight a distinguished basis vector of the tangent space at $p_\mathbf{S}$. 

\begin{prop}\label{Prop:TangentSpaceFicedPoint}
Given a $T$--fixed point $p_\bS$ of $\Fl_{n+1}^a$, the torus $T$ acts on the tangent space at $p_\bS$ diagonally in the basis $\mathcal{B}$. Moreover the eigenvalues are (generically) distinct.
\end{prop}
\begin{proof}
Given $\lambda\in T$ and $f\in {\rm Hom}(p_\mathbf{S}, M/p_\mathbf{S})$, $(\lambda.f)(v)=\lambda.f(\lambda^{-1}.v)$. Now, by definition of $T$, each connected component $R$ of $\theta_M$ has a weight $wt(R)$ and hence a basis vector $(A,f,B)$ receives the weight $wt(B)/wt(A)$.
\end{proof}
To illustrate the previous proposition, let us consider  $\Fl_{4}^a$ and the action of $T$ depicted in the right--hand side of \eqref{Eq:ExTangSpaceBasis}. The tangent space at $p_\mathbf{S}$ has dimension $7$ and the torus acts in the standard basis $\mathcal{B}$ as the diagonal matrix $diag(\frac{1}{\lambda_3}, \frac{\lambda_3}{\lambda_4},\frac{\lambda_2}{\lambda_4}, \frac{\lambda_3}{\lambda_1}, \frac{\lambda_2}{\lambda_1}, \frac{1}{\lambda_2}, \frac{\lambda_4}{\lambda_5})$.  The one--dimensional torus  $T_0$ is given by putting $\lambda_i:=\lambda^i$. In particular its action on the tangent space at $p_\bS$ is given by the diagonal matrix $diag(\lambda^{-3}, \lambda^{-1},\lambda^{-2}, \lambda^{2}, \lambda^{1}, \lambda^{-2}, \lambda^{-1})$. Notice that the eigenvalues of the $T_0$ action are \emph{not} distinct.
\begin{cor}\label{Cor:T-actionTangent}
The $T$--fixed one--dimensional vector subspaces of ${\rm Hom}(p_\mathbf{S}, M/p_\mathbf{S})$ are precisely the coordinate ones, i.e. those generated by standard basis vectors. 
\end{cor}
\subsection{Partial flag varieties.} 
The whole picture described above has a straightforward generalization to the case of partial flag varieties.
Namely, given a collection $\bd=(d_1,\dots,d_k)$, where $1\le d_1<d_2<\dots <d_k\le n$, let $\Fl_\bd$
be the corresponding partial flag variety for $SL_{n+1}$ ($\Fl_\bd$ is a quotient of $SL_{n+1}$ by
a parabolic subgroup). Explicitly, $\Fl_\bd$ consists of collections $(V_{d_1},\dots,V_{d_k})$
of subspaces of an $(n+1)$-dimensional vector space $W$ such that 
$\dim V_m=m$ and $V_{d_i}\subset V_{d_{i+1}}$.
These varieties can be degenerated in the same way as the complete flag variety (see \cite{Fe1},\cite{Fe2}).
As a result one gets a variety $\Fl^a_\bd$, consisting of collections of subspaces 
$(V_{d_1},\dots,V_{d_k})$ of $W$ such that $\dim V_m=m$ and 
\[
pr_{d_i+1}\dots pr_{d_{i+1}}V_{d_i}\subset V_{d_{i+1}},\ i=1,\dots,k-1.   
\]
These varieties are also certain quiver Grassmannians (see \cite{CFR}). Namely, consider
the equioriented quiver of type $A_k$. Then the degenerate partial flag variety $\Fl^a_\bd$ is isomorphic to
\begin{equation}\label{part}
{\rm Gr}_{(d_1,\dots,d_k)} 
(P_1^{d_1}\oplus P_2^{d_2-d_1}\oplus\dots \oplus P_k^{d_k-d_{k-1}}
\oplus I_1^{d_2-d_1}\oplus\dots\oplus I_{k-1}^{d_k-d_{k-1}}\oplus I_k^{n+1-d_k}),
\end{equation}
where $P_i$ and $I_j$ are projective and injective modules of the $A_k$ quiver.
There is a natural surjection $\Fl^a_{n+1}\to \Fl^a_\bd$, sending $(V_i)_{i=1}^n$ to $(V_{d_j})_{j=1}^k$. 
The group $\fA$ thus acts on $\Fl^a_\bd$; the orbits are affine cells containing exactly one $T$-fixed point.
These $T$-fixed points are parametrized by collections $\bS=(S_{d_1},\dots,S_{d_k})$ of subsets of
$\{1,\dots,n+1\}$ subject to the conditions $\#S_{d_i}=d_i$ and 
\begin{equation}\label{dadm}
S_{d_i}\subset S_{d_{i+1}}\cup\{d_i+1,\dots,d_{i+1}\},\ i=1,\dots,k-1.
\end{equation}
We call such collections $\bd$-admissible. 
As for the complete flags, the corresponding $T$-fixed point $p_\bS=(V_{d_1},\dots,V_{d_k})$ is given by 
$V_{d_i}={\mathrm span}(w_j,\ j\in S_{d_i})$.

\section{The moment graph}\label{MG}
In this section we study the combinatorics and geometry of the cellular decomposition of the degenerate 
flag varieties. 

\subsection{The group action.} 
Recall the group ${\mathfrak A}$ acting on $\Fl^a_{n+1}$.
The following lemma is simple, but important for us.
Let $B\subset GL_{2n}$ be the Borel subgroup of lower-triangular matrices and $N\subset B$ be the
subgroup of matrices $(a_{i,j})_{i\ge j}$ such that $a_{i,i}=1$ and $a_{i,j}=0$ unless $i-j>n$.
For example, for $n=5$ the froup $N$ looks as follows:
$$
\left(
\begin{array}{cccccccccc}
1 & 0 & 0 & 0 & 0 & 0 & 0 & 0 & 0 & 0\\
0 & 1 & 0 & 0 & 0 & 0 & 0 & 0 & 0 & 0\\
0 & 0 & 1 & 0 & 0 & 0 & 0 & 0 & 0 & 0\\
0 & 0 & 0 & 1 & 0 & 0 & 0 & 0 & 0 & 0\\
0 & 0 & 0 & 0 & 1 & 0 & 0 & 0 & 0 & 0\\
0 & 0 & 0 & 0 & 0 & 1 & 0 & 0 & 0 & 0\\
* & 0 & 0 & 0 & 0 & 0 & 1 & 0 & 0 & 0\\
* & * & 0 & 0 & 0& 0 & 0 & 1 & 0 & 0\\
* & * & * & 0 & 0 & 0 & 0 & 0 & 1 & 0\\
* & * & * & * & 0 & 0 & 0 & 0 & 0 & 1
\end{array}\right).
$$ 
 
\begin{lem}\label{B}
The group ${\mathfrak A}$ is isomorphic to the quotient group $B/N$. 
\end{lem}
\begin{proof}
Consider the isomorphism ${\rm Aut} (P\oplus I)\simeq {\rm Aut} (\bigoplus_{i=1}^n P_i
\oplus \bigoplus_{k=1}^n I_k)$. We note that  for any pair of indecomposable summands 
of $P\oplus I$ the space of homomorphisms between them is either one-dimensional or trivial. 
More precisely,
let us introduce the following notation for the indecomposable summands of $P\oplus I$:
\begin{equation}\label{phi}
R_1=P_n, R_2=P_{n-1},\dots, R_n=P_1,\ R_{n+1}=I_n, R_{n+2}=I_{n-1},\dots, R_{2n}=I_1. 
\end{equation}
Then for two indecomposable summands
$R_i$ and $R_j$  one has $\dim {\rm Hom} (R_i,R_j)=1$ if and only if
$i\le j$ and $j-i\le n$ (see Lemma \ref{HE}). Hence we obtain a surjection of groups $B\to \fA$
and the kernel coincides with $N$.   
\end{proof}
\begin{rem}
Let us fix a non zero  element $\gamma_{i,j}\in {\rm Hom} (R_i,R_j)$ for each pair $i,j$ with
$i\le j$, $j-i\le n$. Then any element $g\in\fA$ can be uniquely written as a sum 
$\sum g_{i,j}\gamma_{i,j}$, defining a matrix in $B$.
This produces a section $\fA\to B$.    
\end{rem}

\begin{rem}\label{Rem:Ri}
We note that the direct summands $R_i$ in type $A_4$ are visualized in \eqref{la}.
Namely, $R_1$ is represented by the only fat dot in the upper line, $R_2$ is represented by the two dots 
in the next to the upper line, and so on up to $R_8$. In general, if $i\le n$, then the dimension vector of 
$R_i$ is $(0,\dots,0,1,\dots,1)$ with $i$ units and each non zero $(R_i)_k$ is spanned by $w_{n+1-i}$.
If $i>n$, then the dimension vector of 
$R_i$ is $(1,\dots,1,0,\dots,0)$ with $2n-i+1$ units and each non-zero $(R_i)_k$ is spanned by $w_{2n-i+2}$. 
\end{rem}
Recall that the $T$ fixed points in $\Fl^a_{n+1}$ are labeled by the admissible collections. 
For an admissible collection $\bS$ let $p_{\bS}$ be the corresponding $T$-fixed point and $C_{\bS}$
be the cell containing $p_{\bS}$. We know that $C_{\bS}={\mathfrak A} p_{\bS}$. Our goal now
is to describe a unipotent subgroup $U_\bS\subset {\mathfrak A}$ such that the map 
$U_\bS\to C_{\bS}$ is one-to-one. Let ${\mathfrak a}$ be the Lie algebra of the group ${\mathfrak A}$.
Then 
$$\fa={\rm Hom} (P,P)\oplus {\rm Hom} (I,I)\oplus {\rm Hom} (P,I).$$ 
The Lie algebra $\fa$
is the quotient of the Borel subalgebra $\fb\subset\gl_{2n}$ of lower triangular matrices 
by the ideal $\fn$ consisting of matrices $(a_{j,i})_{j\ge i}$ such that $a_{i,j}=0$ unless $j-i>n$
(this is exactly the Lie algebra of $N$).   
In particular, the one-dimensional hom-spaces ${\rm Hom}(R_i,R_j)$, $i\le j$, $j-i\le n$ 
between two indecomposable
summands of $P\oplus I$ correspond to the root vectors of the form 
$E_{j,i}\in\fb$ ($E_{j,i}$ are matrix units). We have
\[
\fa=\ft\oplus\bigoplus_{\genfrac{}{}{0pt}{}{1\le i<j\le 2n}{j-i\le n}} \fa_{i,j}, 
\]  
where $\ft$ is the Lie algebra of the torus $T$ and $\fa_{i,j}={\rm Hom}(R_i,R_j)$.

Consider a pair $R_i,R_j$ of direct summands of $P\oplus I$ such that $\dim {\rm Hom}(R_i,R_j)=1$ 
and fix a non-zero $\gamma\in {\rm Hom}(R_i,R_j)$. 
\begin{dfn}
A pair of indices $(i,j)$ (a pair of representations $R_i,R_j$) is called $\bS$-effective, 
if $p_\bS\cap R_i\ne 0$ and 
$\gamma(p_\bS\cap R_i)$ does not sit inside $p_\bS$.
\end{dfn} 

\begin{rem}\label{Remark:S-effectiveTorus}
$\mathbf{S}$--effective pairs have the following geometric interpretation: they are in bijection with standard basis vectors of the tangent space at $p_\bS$ on which $T_0$ acts with \emph{positive} weight (see the end of subsection~\ref{QGr}). Let us prove this statement. In notation \eqref{phi}, we denote by $\mathbf{R}_k$ the coefficient--quiver of $R_k$. Given an $\bS$--effective pair $(i,j)$ a non--zero $\gamma\in{\rm Hom}(R_i,R_j)$ is determined (up to scalar multiplication) by a (unique) triple $(A,f,B)$. So $A\subset \mathbf{R}_i$ is predecessor--closed, $B\subset \mathbf{R}_j$ is successor closed and $f:A\rightarrow B$ is a quiver isomorphism compatible with $\pi$ (see subsection~\ref{QGr}).
The sub representation $\gamma(p_\bS\cap R_i)\subset R_j$ determines the successor--closed sub quiver $f(\bS\cap A)$ of $B$. Since by definition $\gamma(p_\bS\cap R_i)$ does not sit inside $p_\bS$, $f(\bS\cap A)$ strictly contains $\bS\cap B$ and the difference $f(\bS\cap A)\setminus (\bS\cap B)$ is the coefficient quiver of the non trivial quotient $\gamma(p_\bS\cap R_i)/(\gamma(p_\bS\cap R_i)\cap p_\bS)$. The map
$$
\gamma\mapsto b_\gamma:=(\bS\cap A\setminus f^{-1}(\bS\cap B), f|_{\bS\cap A\setminus f^{-1}(\bS\cap B)}, f(\bS\cap A)\setminus (\bS\cap B))
$$
gives a bijection between $\bS$--effective pairs and standard basis vectors of the tangent space $T_{p_\bS}(\Fl_{n+1}^a)={\rm Hom}(p_\mathbf{S},M/p_\mathbf{S})$
on which $T_0$ acts with a positive weight. To see this we notice that $\bS\cap A$ is predecessor--closed in $\bS$ and $\bS\cap B$ is successor closed in $B$. Then $f^{-1}(\bS\cap B)$ is successor closed in $\bS\cap A$ and hence $\bS\cap A\setminus f^{-1}(\bS\cap B)$ is predecessor closed in $\bS\cap A$ and hence in $\bS$. We notice that $\bS\cap B$ coincides with $\bS\cap \mathbf{R}_j$ (otherwise $\bS\cap B$ would not be strictly contained in $f(\bS\cap A)$). Since $f(\bS\cap A)$ is successor closed in $\mathbf{R}_j$ and $\bS\cap B=\bS\cap \mathbf{R}_j$, it follows that $f(\bS\cap A)\setminus (\bS\cap B)$ is successor closed in $\mathbf{R}_j\setminus (\bS\cap \mathbf{R}_j)$ and hence in $\theta_M\setminus \bS$. The quiver morphsim $f|_{\bS\cap A\setminus f^{-1}(\bS\cap B)}$ is a quiver isomorphism between $\bS\cap A\setminus f^{-1}(\bS\cap B)$ and $f(\bS\cap A)\setminus (\bS\cap B)$ compatible with $\pi$, since $f$ is so. The image $b_\gamma$ of $\gamma$ is hence a standard basis vector of ${\rm Hom}(p_\bs,M/p_\bS)$. The action of $T_0$ on $b_\gamma$ is given by $\lambda. b_\gamma=\lambda^{j-i} b_\gamma$. Since $\gamma\neq 0$, then $i\leq j$ and hence $b_\gamma$ has positive weight. The map is hence well--defined and injective. Let us show that it is surjective. Let $b=(A',f',B')$ be a standard basis vector of ${\rm Hom}(p_\bS,M/p_\bS)$ on which $T_0$ acts with a positive weight. Then there are indices $i$ and $j$ such that $A'$ is a predecessor--closed sub quiver of $\bS\cap \mathbf{R}_i$, and $B'$ is a successor--closed sub quiver of $\mathbf{R}_j\setminus (\mathbf{R}_j\cap\bS)$. The torus $T_0$ acts on $b$ as $\lambda.b=\lambda^{j-i}b$ and hence $j>i$. We claim that $j-i\leq n$. Indeed if $j-i>n$ then $\pi(\mathbf{R}_j)$ and $\pi(\mathbf{R}_i)$ are disjoint in $Q$ (otherwise ${\rm Hom}(R_i,R_j)\neq 0$ against the hypothesis $j-i>n$) and hence the quiver isomorphism  $f':A'\rightarrow B'$ could not exist. In view of Lemma~\ref{HE} and the proof of Lemma~\ref{B}, it follows that there is a non--zero standard basis vector $\gamma\in{\rm Hom}(R_i, R_j)$ defined by a triple $(A, f, B)$. Notice that $\pi(A)=\pi(B)=\pi(R_i)\cap\pi(R_j)\supset \pi(A')=\pi(B')$. It follows that $A'\subset A$, $B'\subset B$ and $f'=f|_{A'}$. From this we conclude that $p_\bS\cap R_i\neq0$ and $\gamma(p_\bS\cap R_i)$ does not sit inside $p_\bS$ and hence $(i,j)$ is an $\bS$--effective pair. 
\end{rem}

Let $U_{i,j}\subset \fA$ be the one-parameter subgroup with the Lie algebra
$\fa_{i,j}$. The importance of effective pairs is explained by the following lemma:
\begin{lem}
If a pair $(i,j)$ is not $\bS$-effective then $U_{i,j} p_{\bS}=p_{\bS}$. Otherwise, the map 
$U_{i,j}\to\Fl^a_{n+1}$, $g\mapsto gp_{\bS}$ is injective. 
\end{lem}
\begin{proof}
Assume that a pair $R_i,R_j$ is not $\bS$-effective and take a non trivial $\gamma\in {\rm Hom}(R_i,R_j)$.
By definition, $\gamma p_{\bS}\subset p_{\bS}$ and hence the exponent of the (scaled) operator 
$\gamma$ fixes $p_{\bS}$. To prove the second claim we note that 
\[
\exp(c\gamma) p_{\bS}=({\mathrm Id} + c\gamma) p_{\bS}.
\] 
Hence, if $\gamma p_{\bS}$ does not sit inside $p_{\bS}$, then all the points
$\exp(c\gamma) p_{\bS}$, $c\in\bC$ are different.
\end{proof}

For an admissible $\bS$ let $\fa_\bS\subset\fa$ be the subspace defined as the direct sum of
one-dimensional spaces ${\rm Hom}(R_i,R_j)$ for all $\bS$-effective pairs $R_i,R_j$.
\begin{lem}
The subspace $\fa_\bS$ is a Lie subalgebra of $\fa$.  
\end{lem}   
\begin{proof}
Let $\gamma_1\in\fa_{i,j}$ and $\gamma_2\in\fa_{k,l}$, $i>j$, $k>l$ be two elements such that
$[\gamma_1,\gamma_2]\ne 0$. Then either $j=k$ or $i=l$. We work out the first case (the second is 
very similar). We have $[\gamma_1,\gamma_2]=\gamma_1\gamma_2\in\fa_{i,l}$. Since $\gamma_2$ 
is $\bS$-effective, we have 
\[
\gamma_2(p_{\bS}\cap R_l)\supsetneq p_{\bS}\cap R_k. 
\] 
Now, since 
\[
\gamma_1(p_{\bS}\cap R_j)\supsetneq p_{\bS}\cap R_i 
\] 
and $j=k$, we obtain that 
\[
\gamma_1\gamma_2(p_{\bS}\cap R_l)\supsetneq p_{\bS}\cap R_i 
\] 
and hence $\gamma_1\gamma_2$ is $\bS$-effective.
\end{proof}

Let $U_\bS$ be the Lie group of $\fa_\bS$, i.e. $U_\bS$ is generated by all $U_{i,j}$ with 
$\bS$-effective $(i,j)$. 
We note that $U_\bS$ is invariant with respect to the  torus $T$
action by conjugation.  
\begin{thm}\label{U}
The map $U_\bS\to C_\bS$, $g\mapsto gp_{\bS}$ is bijective and $T$-equivariant.  
\end{thm}  
\begin{proof}
First, we note that $T$-equivariance follows from $Tp_\bS=p_\bS$. Now let us prove that the map
$U_\bS\to C_\bS$ is surjective. Let us write an element $\g\in\fA$ as $g=g_\bS g_1h$, where
$h\in T$, $g_\bS\in U_\bS$ and $g_1$ belongs to the subgroup of of $\fA$, generated by $U_{i,j}$
with  non $\bS$-effective $(i,j)$. Then $g p_\bS=g_\bS p_\bS$ and hence we are done. Finally, let
us prove the injectivity. Assume that there exists $g\in U_\bS$ such that $gp_\bS=p_\bS$.
We identify $g$ with the corresponding lower triangular matrix in $GL_{2n}$ with enries 
$g_{i,j}$ satisfying $g_{i,i}=1$ and $g_{i,j}=0$ if $i-j>n$. Our goal is to prove that 
$g_{i,j}=0$ for all $i>j$. Let $p(\bS)=(V_1,\dots,V_n)$ and assume that $g_{i,j}\ne 0$ for $i>j$.
Since $g\in U_\bS$, the pair $(i,j)$ is $\bS$-effective. Consider a non-zero element 
$\gamma\in \fa_{i,j}$ (so $\gamma\in{\mathrm Hom}(R_i,R_j)$).
Let $t=1,\dots,n$ be a number such that $V_t\cap R_i\ne 0$ and
$\gamma V_t\cap V_t=0$. Choose a non-zero vector $w\in V_t\cap R_i$. Then $gw\notin V_t$
and hence $gp_\bS\ne p_\bS$.      
\end{proof}

\begin{rem}
We note that Theorem \ref{U} is analogous to the corresponding theorem for classical flag varieties, see
e.g. \cite{T}, Lemma $3.2$.  
\end{rem}

\begin{prop}\label{3}
The number of $\bS$-effective pairs $(i,j)$ is equal to the sum $N_{PI}(\bS)+N_{PP}(\bS)+N_{II}(\bS)$ 
of three numbers defined by:
\begin{itemize}
\item $N_{PI}(\bS)$ is the number of pairs $1\le k<l\le n+1$ such that there exists $t$ with 
$k\le t<l$ such that $k\in S_t$, $l\notin S_t$.
\item $N_{PP}(\bS)$ is the number of pairs $1\le k<l\le n$ such that there exists $t\ge l$ such 
that $l\in S_t$, $k\notin S_t$.
\item $N_{II}(\bS)$ is the number of pairs $2\le k<l\le n+1$ such that there exists $t<k$ such that
$l\in S_t$, $k\notin S_t$.
\end{itemize} 
\end{prop}
\begin{proof}
We divide $\bS$-effective pairs into three parts $R_i,R_j\subset P$, $R_i,R_j\subset I$ and
$R_i\subset P, R_j\subset I$.
We claim that the number of $\bS$-effective pairs from the first (second, third) part is equal to $N_{PP}(\bS)$ ($N_{II}(\bS)$, $N_{PI}(\bS)$). 
\begin{enumerate}
\item The case $R_i\subset P$, $R_j\subset I$. Then $1\leq i\leq n<j\leq 2n$. Since $(i,j)$ is $\bS$--effective, there exists an index $t:\, n+1-i\leq t\leq 2n+1-j$ such that $n+1-i\in S_t$ 
and $2(n+1)-j\notin S_t$. Put $k=n+1-i$ and $l=2(n+1)-j$.
\item The case $R_i,R_j\subset P$. Since $(i,j)$ is $\bS$--effective then $1\leq i< j\leq n$ and there is 
an index $t:\, t\geq n+1-i> n+1-j$ such that  $n+1-i\in S_t$ and $n+1-j\notin S_t$. 
Put $l=n+1-i$ and $k=n+1-j$. 
\item The case $R_i,R_j\subset I$. Since $(i,j)$ is $\bS$--effective then $n+1\leq i< j\leq 2n$ and 
there is an index $t:\, t\leq 2n+1-j< 2n+1-i$ such that  $2(n+1)-i\in S_t$ and $2(n+1)-j\notin S_t$. 
Put $l=2(n+1)-i$ and $k=2(n+1)-j$.
\end{enumerate}
\end{proof}

\begin{cor}\label{Cor:CellDim}
The dimension of $C_{\bS}$ is equal to the sum $N_{PI}(\bS)+N_{PP}(\bS)+N_{II}(\bS)$.
\end{cor} 
\begin{proof}
Thanks to Theorem \ref{U} the dimension of the cell $C_{\bS}$ is equal to the number of 
$\bS$-effective pairs $R_i$, $R_j$. Now Proposition \ref{3} implies the corollary.
\end{proof}

\begin{cor}\label{Cor:PoincPoly}
The Poincar\'e polynomial of $\Fl^a_{n+1}$ is equal to the sum of the terms 
$q^{N_{PI}(\bS)+N_{PP}(\bS)+N_{II}(\bS)}$, where the sum runs over the set of admissible collections.
\end{cor}

\begin{rem}
In \cite[Theorem 5.1]{CFR} it is shown that although $\Fl_{n+1}^a$ is not smooth, the one--dimensional sub torus $T_0$ of $T$ still produces a Bia\l ynicki--Birula type cell decomposition (\cite{BB}, \cite[Theorem~2.4.3]{CG}). In other words, the attracting set of a $T_0$--fixed point $p_\bS$ is a cell and it has dimension equal to the dimension of the positive part of the tangent space at $p_\bS$ 
(the positive part is the vector subspace generated by vectors on which $T_0$ acts with positive weight). 
In view of Remark \ref{Remark:S-effectiveTorus}, this dimension is precisely the number of $\mathbf{S}$--effective pairs. Theorem \ref{U} provides another and more explicit proof of this fact. 
\end{rem}
\begin{rem}\label{Rem:CellDimDiagram}
From the discussion above (see Corollary~\ref{Cor:CellDim} and Remark~\ref{Remark:S-effectiveTorus}), the dimension of the cell with center $p_\bS$ can be easily read off from $\bS$, viewed inside the coefficient quiver of $P\oplus I$ written in the form \eqref{la}. Indeed in this diagram let us color a vertex black if it belongs to $\bS$ and white otherwise. In the $i$--th columns (counting from left to right) there are precisely $i$ black vertices. Some of them are sources of $\bS$. For every such source $t\in S_i$ let us count the number $w_t$ of white vertices below it. Let $c_i$ be the sum of the $w_t$'s. Then the dimension of the cell with center $p_\bS$ equals the sum $c_1+c_2+\cdots+c_n$. For example let us consider the following $T$--fixed point of $\Fl_5^a$:
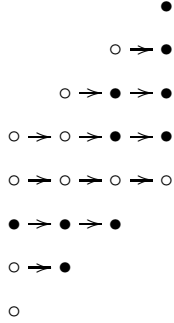
$$
\xymatrix@R=6pt@C=8pt
{
&&&&\bullet&\\
&&&\circ\ar[r]&\bullet&\\
&&\circ\ar[r]&\bullet\ar[r]&\bullet&\\
&\circ\ar[r]&\circ\ar[r]&\bullet\ar[r]&\bullet&\\
&\circ\ar[r]&\circ\ar[r]&\circ\ar[r]&\circ&\\
&\bullet\ar[r]&\bullet\ar[r]&\bullet&&\\
&\circ\ar[r]&\bullet&&&\\
&\circ&&&&\\
}
$$
then $c_1=2$, $c_2=0$, $c_3=2$ and $c_4=2$. The cell has hence dimension $6$.
\end{rem}
\subsection{Moment graph.}
We briefly recall the definition of a moment graph  (see \cite{BM},\cite{GKM}). 
Let $X$ be a projective algebraic
variety acted upon by a torus $T=(\bC^*)^d$ with a fixed one-dimensional subtorus $\imath:\bC^*\subset T$.
Assume that the $T$ action on $X$ has finitely many fixed points and one-dimensional orbits and 
any $\bC^*$ fixed point is $T$-fixed ($X^T=X^{\bC^*}$). Assume further that $X$ has a decomposition
as a disjoint union of $T$-invariant affine cells in such a way that each cell $C$ contains exactly 
one $\bC^*$-fixed point $p$ and $C=\{x\in X:\ \lim_{\lambda\to 0} \imath(\la)x=p\}$ (i.e. the cell consists
of all  points attracted by $p$, see \cite{BB}). We denote this cell by $C_p$. 
The moment graph $\Gamma$ has its set of vertices
labeled by the $T$-fixed points. Two points $p_1$ and $p_2$ are connected by an edge in $\Gamma$ if there exists 
a one-dimensional $T$-orbit $L$ such that $\bar L=L\sqcup p_1\sqcup p_2$ (i.e. $p_1$ and $p_2$ are
the $T$-fixed points in the closure of $L$). Thus the edges of $\Gamma$ are labeled by the 
one-dimensional $T$-orbits. We orient  $\Gamma$ by the following rule: for two vertices $p_1$ and $p_2$
we say $p_1\ge p_2$ if $C_{p_2}\subset \bar C_{p_1}$. If there is an edge connecting $p_1$ and $p_2$
in $\Gamma$ then we put an arrow $p_1\to p_2$. Finally, one defines a labeling $\al_L$ of the edges $L$ 
of $\Gamma$
by the elements $\al_L\in \ft^*$, where $\ft$ is the Lie algebra of the torus $T$. Namely, for an edge 
$L$ let $T_x\subset T$ be the stabilizer of a point $x\in L$ (obviously, $T_x$ is independent of $x\in L$).
Then the Lie algebra $\ft_x\subset \ft$ is a hyperplane. We define $\al_L$ as a non-zero element in the annihilator of $\ft_x$.    
 
\begin{example}
Here we give an example of the moment graph for the classical flag variety $\Fl_3=SL_3/B$. The torus
$T$ has $6$ fixed points labeled by pairs $(S_1,S_2)$ of subsets of $\{1,2,3\}$ such that $\#S_1=1$,
$\#S_2=2$ and $S_1\subset S_2$. The moment graph of $\Fl_3$ looks as follows:
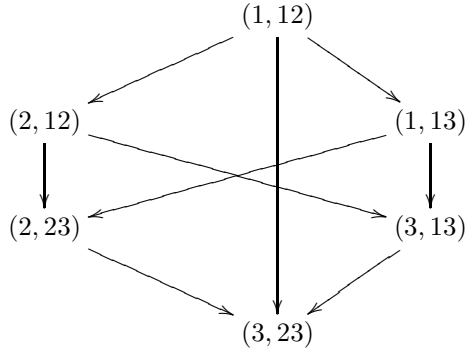
\[
\xymatrix{
& & (1,12)\ar[dll]\ar[dr]\ar[ddd] &\\
(2,12) \ar[d] \ar[drrr]& & & (1,13)\ar[d] \ar[dlll]\\
(2,23) \ar[drr] & & & (3,13)\ar[dl] \\
& & (3,23) &  \\
}
\]
We note that usually the arrows in the moment graph direct from bottom to top. However for
our purposes it is more convenient to draw the vertices from top to bottom, since 
in the degenerate situation the dense cell corresponds to the point $(1,12)$, see Example
\ref{A2}. This is not important in the classical situation
due to the Chevalley involution, but crucial in the degenerate case.  
\end{example}

Our goal is to describe the moment graph of the degenerate flag varieties. 
\begin{rem}
We note that the moment graphs turn out to be a powerful tool for computing various cohomology groups 
of algebraic varieties (see \cite{BM}, \cite{GKM}, \cite{T}, \cite{Fi}, \cite{FW}). 
A crucial role is played by the notion of
sheaves on moment graphs. In this paper we do not discuss $\Gamma$-sheaves, but only describe the 
combinatorial structure of the graphs. Computation of the (equivariant) cohomology
as well as the (equivariant) intersection cohomology of the degenerate flag varieties is an interesting
open problem.          
\end{rem}

\begin{example}\label{A2}
Here we give a picture of the moment graph for the degenerate flag variety $\Fl^a_3$. Recall that
the $T$-fixed points are labeled by pairs $(S_1,S_2)$ of subsets of the set $\{1,2,3\}$ such that
$\#S_1=1$, $\#S_2=2$ and $S_1\subset S_2\cup\{2\}$. 
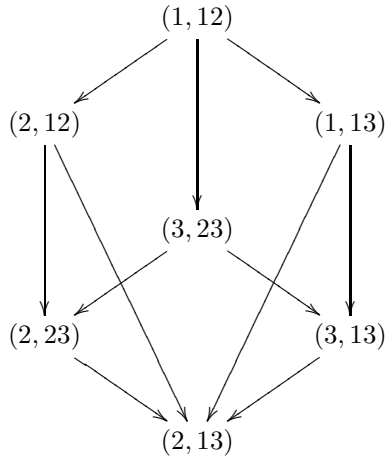
\[
\xymatrix{
& (1,12)\ar[dl]\ar[dr]\ar[dd] &\\
(2,12)\ar[dd]\ar[dddr] & & (1,13)\ar[dd]\ar[dddl]\\
& (3,23)\ar[dr]\ar[dl]  &\\
(2,23)\ar[dr] & & (3,13)\ar[dl]\\
& (2,13)&
}
\]
\end{example} 

The moment graph for the degenerate flag variety $\Fl^a_4$ is computed in Appendix~C.

We now give an explicit combinatorial description of the moment graph.
We identify the Lie algebra $\ft$ of $T$ with the diagonal traceless $2n\times 2n$ matrices.
For a pair of indices $i,j$, $1\le i<j\le 2n$, we denote by $\al_{i,j}\in\ft^*$ the element
$\al_{i,j}({\mathrm diag}(x_1,\dots,x_{2n}))=x_i-x_j$. 

\begin{thm}\label{Thm:OneDimOrbits}
The number of one-dimensional $T$-orbits in $\Fl^a_{n+1}$ is finite. 
The orbits are of the form $U_{i,j} p_{\bS}\setminus p_{\bS}$,
where $\bS$ is admissible and $(i,j)$ is $\bS$-effective. The edge in the moment
graph, which corresponds to $U_{i,j} p_{\bS}\setminus p_{\bS}$ is labeled by $\al_{i,j}$.   
\end{thm}
\begin{proof}
Thanks to Theorem \ref{U}, we only need to describe the one-dimensional $T$-orbits in $U_\bS$.
It is easy to see that these are non-identity elements in $U_{i,j}$. 
\end{proof}
\begin{rem}
Theorem~\ref{Thm:OneDimOrbits} also follows from Corollary~\ref{Cor:T-actionTangent} and Remark~\ref{Remark:S-effectiveTorus}. Indeed in view of Corollary~\ref{Cor:T-actionTangent}, the directions around $p_\bS$ of the one--dimensional $T$--orbits containing $p_\bS$ are precisely the standard basis vectors of the tangent space $T_{p_\bS}(\Fl_{n+1}^a)$ at $p_\bS$. In particular the number of such $T$--orbits is bigger or equal than ${\rm dim} \Fl_{n+1}^a$ and it is equal if and only if $p_\mathbf{S}$ is smooth. Any such curve $\ell$ consists of three T--orbits $\ell=\{p_\mathbf{S}\}\cup\{\ell'\}\cup\{p_\mathbf{R}\}$. 
The direction of $\ell$ is fixed also by the one--dimensional torus $T_0$. In particular this standard basis vector of  $T_{p_\mathbf{S}}(\Fl_{n+1}^a)$ has either positive or negative $T_0$ weight. If the weight is positive then $\{p_\mathbf{S}\}\cup\{\ell'\}$ sits inside the attracting set of $p_\mathbf{S}$ which is the cell ${\mathfrak A}p_\mathbf{S}$ and hence $p_\mathbf{R}$ (and its attracting cell) is in the closure of this cell. It follows that in the moment graph there is an arrow $p_\mathbf{S}\rightarrow p_\mathbf{R}$. In particular the number of arrows starting from $p_\bS$ in the moment graph, equals the number of standard basis vector of $T_{p_\bS}(\Fl_{n+1}^a)$ on which $T_0$ acts with positive weight. In view of Remark~\ref{Remark:S-effectiveTorus} this number equals the number of $\bS$--effective pairs.
\end{rem}

\begin{cor}
The dimension of a cell $C_{\bS}$ is equal to the number of edges in the moment graph which are
directed outwards the vertex $p_{\bS}$.     
\end{cor}

The following theorem generalizes the results as above to the case
of  the degenerate partial flag varieties. 
\begin{thm}
The number of one-dimensional $T$ orbits on $\Fl^a_\bd$ is finite. Each of these orbits
is covered by a one-dimensional $T$-orbit in $\Fl^a_{n+1}$ via the surjection $\Fl^a_{n+1}\to\Fl^a_\bd$.
All the orbits are of the form $U_{i,j}p\setminus p$ for some $i,j$ and a $T$-fixed $p\in \Fl^a_\bd$.  
\end{thm} 

\section{Smooth locus and the Schr\"oder numbers}
In this section we describe the smooth  locus of the degenerate flag varieties
$\Fl^a_{n+1}$ and compute Euler characteristics and Poincar\'e polynomials.
\subsection{Smooth cells.} 
Take a point $N\in {\rm Gr}_{\dimv P}(P\oplus I)$. Then $N$ can be split as $N=N_P\oplus N_I$, where
$N_P\subset P$ and $N_I\subset I$, such that $N_I$ and $P/N_P$ are of the same dimension vector (see \cite[Theorem 1.3]{CFR}). 
\begin{lem}\label{T}
A point $N$ in a quiver Grassmannian ${\rm Gr}_{\dimv P}(P\oplus I)$ is smooth if and only if 
${\rm Ext}^1(N_I,P/N_P)=0$.
\end{lem}
\begin{proof}
Let $\langle\cdot,\cdot\rangle$ be the Euler form of the quiver $Q$, given on a pair of dimension vectors ${\bf d}$, ${\bf e}$ by $\langle {\bf d},{\bf e}\rangle=\sum_{i=1}^n d_ie_i-\sum_{i=1}^{n-1} d_ie_{i+1}$. Then $\langle\dimv X,\dimv Y\rangle=\dim {\rm Hom}(X,Y)-\dim{\rm Ext}^1(X,Y)$ for arbitrary representations $X$ and $Y$ of $Q$. By \cite[Theorem 1.1]{CFR}, we have
$$\langle\dimv P,\dimv I\rangle=\dim {\rm Gr}_{\dimv P}(P\oplus I).$$ 
By the formula \cite[Lemma 2.3]{CFR} for the dimension of the tangent space $T_N$ to the point $N\in{\rm Gr}_{\dimv P}(P\oplus I)$, we then have
\begin{multline*}
\dim T_N=\dim{\rm Hom} (N_I\oplus N_P, P/N_P\oplus I/N_I)=\\
\langle\dimv P,\dimv I\rangle - \dim {\rm Ext}^1 (N_I\oplus N_P, P/N_P\oplus I/N_I).
\end{multline*}
Since $N_P$ is projective and $N/N_I$ is injective, we obtain
\[
\dim {\rm Ext}^1 (N_I\oplus N_P, P/N_P\oplus I/N_I)=\dim {\rm Ext}^1 (N_I, P/N_P).
\] 
Hence, the dimension of the tangent space at a point $N$ is equal to the dimension of the Grassmannian if 
and only if ${\rm Ext}^1 (N_I, P/N_P)$ vanishes.  
\end{proof}

Recall that the quiver Grassmannian ${\rm Gr}_{\dimv P}(P\oplus I)$ can be decomposed into the disjoint
union of ${\mathfrak A}$-orbits of the form ${\mathfrak A} p_{\bS}$. Hence all the points of the orbit are smooth or singular
together with $p_{\bS}$. So it suffices to understand what are the conditions for an admissible collection
$\bS$ that guarantee the smoothness of $p_{\bS}$. We use Lemma \ref{T} above.

\begin{thm}\label{main}
A point $p_{\bS}$ is smooth if and only if for all $1\leq j<i\leq n$, the condition 
$i\in S_{j}$ implies $j+1\in S_i$.
\end{thm}
\begin{proof}
Given an admissible collection $\bS=(S_i)_{i=1}^n$, we introduce the following numbers for 
all $i=1,\ldots,n+1$:
$$k_i=\min\{1\leq k<i\, :\, i\in S_k\},\;\;\; l_j=\min\{j\leq l\leq n\, :\, j\in S_l\}.$$
Recall the indecomposable representations $R_{k,l}$  with 
the support on the interval $[k,l]$.
A representation $p_{\bS}$ is isomorphic to the direct sum $N_I\oplus N_P$, where
$N_I\subset I$ and $N_P\subset P$. It is easy to see that
$$N_I=\bigoplus_i R_{k_i,i-1},\;\;\; P/N_P=\bigoplus_j R_{j,l_j-1}.$$
The extension groups between the indecomposables are given by Lemma \ref{HE}.
Thus we obtain that $0\not={\rm Ext}^1(N_I,P/N_P)$ if and only if there exist indices 
$i$ and $j$ such that 
$k_i+1\leq j\leq i\leq l_j-1$. This holds (writing out the three inequalities) if and only if there exist 
indices $j\leq i$  such that
$$\min\{1\leq k<i\, :\, i\in S_k\}<j,\;\;\; \min\{j\leq l\leq n\, :\, j\in S_l\}>i.$$
This translates into the condition that there exist $j\leq i$ such that $i\in S_{j-1}$, but $j\not\in S_i$.
Conversely, this means that the orbit is smooth if and only if for all $1\leq j\leq i\leq n+1$, if 
$i\in S_{j-1}$, then $j\in S_i$. Note that this condition is void in case $j=1$ or $i=n+1$, so that we can 
replace $j$ by $j-1$, and obtain the assertion of theorem. 
\end{proof}

In what follows we call an admissible collection $\bS$ smooth iff $p_{\bS}$ is a smooth point.

\subsection{The large Schr\"oder numbers.}
Let $r_n$ be the $n$-th large Schr\"oder number, defined as the number of Schr\"oder paths, i.e.
subdiagonal lattice paths from $(0,0)$ to $(n,n)$ consisting of the steps $(0,1)$, $(1,0)$ or $(1,1)$.
The sequence $r_0,r_1,r_2,...$ starts with $1,2,6,22,90,394$. Here are the six Schr\"oder paths for $n=2$:
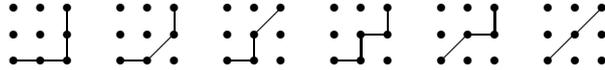
\[
\begin{picture}(40,30)
\multiput(0,0)(10,0){3}{\circle*{3}}
\multiput(0,10)(10,0){3}{\circle*{3}}
\multiput(0,20)(10,0){3}{\circle*{3}}
\put(0,0){\line(1,0){20}}
\put(20,0){\line(0,1){20}}
\end{picture}
\begin{picture}(40,30)
\multiput(0,0)(10,0){3}{\circle*{3}}
\multiput(0,10)(10,0){3}{\circle*{3}}
\multiput(0,20)(10,0){3}{\circle*{3}}
\put(0,0){\line(1,0){10}}
\put(10,0){\line(1,1){10}}
\put(20,10){\line(0,1){10}}
\end{picture}
\begin{picture}(40,30)
\multiput(0,0)(10,0){3}{\circle*{3}}
\multiput(0,10)(10,0){3}{\circle*{3}}
\multiput(0,20)(10,0){3}{\circle*{3}}
\put(0,0){\line(1,0){10}}
\put(10,0){\line(0,1){10}}
\put(10,10){\line(1,1){10}}
\end{picture}
\begin{picture}(40,30)
\multiput(0,0)(10,0){3}{\circle*{3}}
\multiput(0,10)(10,0){3}{\circle*{3}}
\multiput(0,20)(10,0){3}{\circle*{3}}
\put(0,0){\line(1,0){10}}
\put(10,0){\line(0,1){10}}
\put(10,10){\line(1,0){10}}
\put(20,10){\line(0,1){10}}
\end{picture}
\begin{picture}(40,30)
\multiput(0,0)(10,0){3}{\circle*{3}}
\multiput(0,10)(10,0){3}{\circle*{3}}
\multiput(0,20)(10,0){3}{\circle*{3}}
\put(0,0){\line(1,1){10}}
\put(10,10){\line(1,0){10}}
\put(20,10){\line(0,1){10}}
\end{picture}
\begin{picture}(40,30)
\multiput(0,0)(10,0){3}{\circle*{3}}
\multiput(0,10)(10,0){3}{\circle*{3}}
\multiput(0,20)(10,0){3}{\circle*{3}}
\put(0,0){\line(1,1){10}}
\put(10,10){\line(1,1){10}}
\end{picture}
\]

We note that (see e.g. \cite{BEK},\cite{BSS},\cite{St})
\[
r_n=r_{n-1}+\sum_{k=0}^{n-1} r_kr_{n-1-k}.
\]  
The small Schr\"oder numbers $s_n$ are defined as halves of the large ones.

Recall that a collection $\bS=(S_a)_{a=1}^n$ of subsets of the set  $ \{1,\dots,n+1\}$ is smooth 
if  $\#S_a=a$, $S_a\subset S_{a+1}$ and  
for all $1\le a< b\le n$ the following condition holds (see Theorem \ref{main}): 
\[
\text{if } b\in S_a, \text{ then } a+1\in S_b. 
\]   
Let $LS_n$ be the set of length $n$ smooth collections and $\bar r_n$ be the cardinality of $LS_n$.
\begin{prop}
The numbers $\bar r_n$ satisfy the recursion 
$$\bar r_n=\bar r_{n-1}+\sum_{k=0}^{n-1}\bar r_k\bar r_{n-1-k}.$$
\end{prop}
\begin{proof}
We divide all smooth collections according to the values of $S_1$. So first, let $S_1=\{1\}$.
Let us show that the number of such smooth collections is equal to $\bar r_{n-1}$.
Note that all $S_a$ contain $1$. For $a=1,\dots,n-1$ we set 
\[
S'_a=\{i:\ i+1\in S_{a+1}\}\subset \{1,\dots,n\}.
\]
We claim that the collection $(S'_a)_{a=1}^{n-1}$ is smooth and all (length $n-1$) smooth collections arise
in this way. First, obviously, $\#S'_a=a$ and $S'_a\subset S'_{a+1}$. Now the conditions
($b\in S_a$ implies $a+1\in S_b$),  $2\le a<b\le n$ are equivalent to the conditions
($b\in S'_a$ implies $a+1\in S'_b$),  $1\le a<b\le n-1$.

Let $LS_n^k\subset LS_n$ be the set of smooth collections satisfying $S_1=\{k\}$, $2\le k\le n+1$. 
We want to show that the cardinality of $LS^k_n$ is equal to $\bar r_{k-2}\bar r_{n+1-k}$. 
To this end we construct a bijection  $F:LS_n^k\to LS_{k-2}\times LS_{n+1-k}$.
For convenience, we write $F=(f,g)$, where
\[
f:LS_n^k\to LS_{k-2},\ g: LS_n^k\to  LS_{n+1-k}.
\]

First, since $S_1\subset S_a$ for
any $a$, we have $k\in S_a$, $2\le a\le n$. Now the conditions $k\in S_a$ for $a=1,\dots,k-1$ imply
\[
\{2,\dots,k\}\subset S_a \text{ for all } a\ge k.
\]
Given a collection $\bS\in LS_n^k$, we define
\[
g(\bS)=(g(\bS)_1,\dots,g(\bS)_{n+1-k})
\]     
as follows: 
\begin{equation}\label{g}
g(\bS)_a=
\begin{cases}
\{i:\ 2\le i\le n+1-k,\ i+k-1\in S_{a+k-1}\},\text{ if } 1\notin S_{a+k-1},\\
\{1\}\cup \{i:\ 2\le i\le n+1-k,\ i+k-1\in S_{a+k-1}\}\text{ otherwise}.
\end{cases}
\end{equation}
We note that the image depends only on the sets $S_a$ with $a\ge k$. 

Now, given a collection $\bS\in LS_n^k$, we need to define
\[
f(\bS)=(f(\bS)_1,\dots,f(\bS)_{k-2}).
\]     
Let $S_k=\{2,\dots,k\}\cup \{i\}$ for some $i=1,k+1,\dots,n+1$. We note that since $k\in S_a\subset S_k$
for $a<k$, each $S_a\setminus k$ for $2\le a\le k-1$ is an $(a-1)$-element subset of the fixed set of 
cardinality $k-1$ (this set is $\{2,\dots,k-1\}\cup\{i\}$). We now define the map $f$ as follows:
\begin{equation}\label{f}
f(\bS)_a=
\begin{cases}
\{i:\ 1\le i\le k-2,\ i+1\in S_{a+1}\},\text{ if } S_{a+1}\subset \{2,\dots,k\},\\
\{i:\ 1\le i\le k-2,\ i+1\in S_{a+1}\}\cup \{k-1\}, \text{ otherwise}.
\end{cases}
\end{equation}
We note (this is important) that $f_1(\bS)$ depends only on $S_2,\dots,S_{k-1}$.

Our first goal is to show that $f(\bS)\in LS_{k-2}$ and $g(\bS)\in LS_{n-k+1}$ for any $\bS\in LS_n^k$.
By definition, $g(\bS)_a\subset g(\bS)_{a+1}$ for $1\le a\le n-k$ and
\[
g(\bS)_a\in \{1,\dots,n-k+2\}, \ \#g(\bS)_a=a \text{ for } 1\le a\le n-k+1.
\]    
Let us show that for $1\le a<b\le n-k+1$ the inclusion $b\in g(\bS)_a$ implies
$a+1\in g(\bS)_b$. Since $b>1$, $b\in g(\bS)_a$ implies $b+k-1\in S_{a+k-1}$. Since $\bS$ is smooth, 
we obtain $a+k\in S_{b+k-1}$, which gives $a+1\in g(\bS)_b$ and we are done. 
Similarly, one proves that $f(\bS)\in LS_{k-2}$.

Finally, we have to prove that the map $F=(f,g):LS_n^k\to LS_{k-2}\times LS_{n-k+1}$ is one-to-one.
Given an element $(\bS',\bS'')\in LS_{k-2}\times LS_{n-k+1}$, we use formulas \eqref{g} and
\eqref{g} to reconstruct $\bS$ such that $F(\bS)=(\bS',\bS'')$.  
\end{proof}

\begin{cor}
The Euler characteristic of the smooth locus of $\Fl^a_{n+1}$ is equal to the $n$-th 
Schr\"oder number $r_n$.
\end{cor}

Finally, let us formulate the analogue of Theorem \ref{main} for the degenerate partial flag varieties.
We omit the proof since it is very close to the proof of Theorem \ref{main}.
Recall that the $T$-fixed points in $\Fl^a_\bd$ are labeled by $\bd$-admissible collections 
$\bS=(S_{d_1},\dots,S_{d_k})$ (see \eqref{dadm}).
\begin{thm}
A $T$-fixed point $p_\bS\in \Fl^a_\bd$ is smooth if and only if the following conditions hold:
if $b\in S_{d_i}$ and $d_{j+1}\ge b>d_j$ for some $j\ge i$, then 
$\{d_i+1,\dots,d_{i+1}\}\subset S_{d_{j+1}}$.  
\end{thm}

\subsection{Poincar\'e polynomials}
There are several ways to define $q$-analogues of the Schr\"oder numbers (see 
\cite{BDPP}, \cite{BSS},\cite{BEK}). 
We will need the 
simplest one (see \cite{BSS}, page $37$, polynomials $d_n(q)$). They are  called Narayana polynomials 
there, but
in other papers the same polynomials are also referred to as the Schr\"oder polynomials, see 
e.g. \cite{G}).
For a Schr\"oder path $P$ let $diag(P)$ be the number of the diagonal steps in $P$.
Define $r_n(q)$ as the sum of the terms $q^{diag(P)}$ over the set of Schr\"oder paths $P$.
Here are the first several polynomials 
\begin{gather*}
r_0(q)=1,\qquad  r_1(q)=1+q,\qquad r_2(q)=2+3q+q^2,\\
r_3(q)=5+10q+6q^2+q^3,\qquad r_4(q)=14+35q+30q^2+10q^3+q^4.
\end{gather*}
Clearly, $r_n(0)$ is the $n$-th Catalan number.
Let $P_n^{sm}(q)$ be the Poincar\'e polynomial of the smooth locus of $\Fl^a_{n+1}$. 
Our goal here is to prove the following theorem:
\begin{thm}\label{Pnsm}
$P_n^{sm}(q)=q^{n(n-1)/2}r_n(q).$
\end{thm} 

Recall (see \cite{BEK}, \cite{BSS}) that 
\begin{equation}
\label{r(q)}
r_n(q)=qr_{n-1}(q)+\sum_{k=0}^{n-1} r_k(q)r_{n-1-k}(q).
\end{equation}

\begin{prop}\label{PP}
The Poincar\'e polynomials of the smooth locus satisfy the following recursion:
\begin{equation}\label{qS}
P_n^{sm}(q)=q^nP_{n-1}^{sm}(q) + \sum_{l=0}^{n-1} q^{(l+1)(n-l)-1} P_l^{sm}(q)P_{n-1-l}^{sm}(q).
\end{equation}
\end{prop}
\begin{proof}
First, let us consider a smooth collections  $(S_1,\dots,S_n)$ with $S_1=1$. Then the cells 
labeling such collections are in
one-to-one correspondence with smooth collections $\bS'$ of length $n-1$: $S_i'=S_{i+1}\setminus \{1\}$.
We claim that 
\begin{equation}\label{k=1}
\dim C_{\bS}= \dim C_{\bS'} + n.
\end{equation}
We use Proposition \ref{3}. Clearly, the terms $N_{PP}$ and $N_{II}$ for $p_{\bS}$ and $p(\bS')$ do
coincide and the difference of the terms $N_{PI}$ is equal to $n$ (since $S_1=\{1\}$, in the definition
of $N_{PI}$ we can take $t=1$, $i=1$, $j=2,\dots,n+1$). Now \eqref{k=1} produces the first term of
the right hand side of \eqref{qS}.   

Recall  the bijection 
$F=(f,g):LS_n^k\to LS_{k-2}\times LS_{n-k+1}$, $k\ge 2$,
from the set of smooth collections with $S_1=\{k\}$ to the product $LS_{k-2}\times LS_{n-k+1}$. 
Our goal is to prove that
\begin{equation}\label{k}
\dim C_{\bS}= \dim C_{f(\bS)} + \dim C_{g(\bS)} + (k-1)(n+2-k)-1
\end{equation}
(after the shift $l=k-2$ one gets the corresponding term in \eqref{qS}). Recall that since
$k\in S_1$ we have
\[
\{2,\dots,k\}\subset S_m \text{ for all } m\ge k.
\]  
In particular, $S_k=\{2,\dots,k\}\cup \{r\}$ for some number $r=1,k+1,\dots,n+1$.
We claim that
\begin{multline*}
N_{PI}(\bS)+ N_{PP}(\bS)=\\
= N_{PI}(f(\bS)) + N_{PI}(g(\bS)) + N_{PP}(f(\bS))+ N_{PP}(g(\bS)) + (k-1)(n+1-k),
\end{multline*}
and
\[ 
N_{II}(\bS)=N_{II}(f(\bS)) + N_{II}(g(\bS)) + k-2.
\]
First, let us prove the first formula. Assume that $1=r=S_k\setminus\{2,\dots,k\}$. Then
\begin{gather*}
N_{PI}(\bS)= N_{PI}(f(\bS)) + N_{PI}(g(\bS))+ (k-1)(n+1-k),\\ 
N_{PP}(\bS)= N_{PP}(f(\bS)) + N_{PP}(g(\bS)).
\end{gather*}
Here the term $(k-1)(n+1-k)$ comes from the fact that in the definition of $N_{PI}(\bS)$ one can
take $i=2,\dots,k$, $j=k+1,\dots,n+1$ and $t=k$. These possibilities are not counted in
$N_{PI}(f(\bS)) + N_{PI}(g(\bS))$. Now assume that $r>k$. Then one has
\begin{gather*}
N_{PI}(\bS)= N_{PI}(f(\bS)) + N_{PI}(g(\bS))+ (k-1)(n-k),\\ 
N_{PP}(\bS)= N_{PP}(f(\bS)) + N_{PP}(g(\bS)) + k-1.
\end{gather*}
Here  the term $(k-1)(n-k)$ comes from the fact that in the definition of $N_{PI}(\bS)$ one can
take $i=2,\dots,k$, $j\in\{k+1,\dots,n+1\}\setminus r$ and $t=k$.
The term $k-1$ in the right hand side of the second equality  comes from the fact that in the definition 
of $N_{PP}(\bS)$ one can take $i=1$, $j=2,\dots,k$ and $t=k$. All these possibilities are lost 
when computing $N_{PI}(f(\bS))$, $N_{PI}(g(\bS))$,  $N_{PP}(f(\bS))$ and $N_{PP}(g(\bS))$.

Now let us prove that  
\[ 
N_{II}(\bS)=N_{II}(f(\bS)) + N_{II}(g(\bS)) + k-2.
\]
Here the argument is even simpler: the missing $k-2$ comes from the following possibilities for
$N_{II}(\bS)$ missing in $N_{II}(f(\bS)) + N_{II}(g(\bS))$: $i=2,\dots,k-1$, $j=k$, $t=1$.

We thus obtain
\[
\dim C_{\bS}= \dim C_{f(\bS)} + \dim C_{g(\bS)} + (k-1)(n+1-k) + (k-2),
\] 
which implies \eqref{k} and as well as the proposition.
\end{proof}

\begin{cor}
Theorem \ref{Pnsm} holds.
\end{cor}
\begin{proof}
We note that $P_1^{sm}(q)=1+q=r_1(q)$. Now the induction procedure combined with 
\eqref{r(q)} and Proposition \ref{PP} gives the desired result.
\end{proof}

\begin{rem}
It is natural to define a $q,t$-version $h_{n+1}(q,t)$ of the normalized median Genocchi numbers as the sum
over admissible collections $\bS$ of the terms 
$$q^{\dim C_{\bS}}t^{\dim T_{p_{\bS}} \Fl^a_{n+1}}t^{-n(n+1)/2}.$$
Then the value $h_n(1,1)$ is exactly the normalized median Genocchi number and
$h_{n+1}(q,0)=q^{n(n-1)/2}r_n(q)$ is the (scaled) $n$-th Schr\"oder polynomial. 
Here are first few $q,t$-Genocchi polynomials:
\begin{gather*}
h_2(q,t)=1+q,\quad h_3(q,t)=2q+3q^2+q^3 + t,\\
 h_4(q,t)=q^3(5+10q+6q^2+q^3)+tq(2q+7q^2+5q^3)+t^2(1+q).
\end{gather*}    
\end{rem}

\subsection{Schr\"oder numbers: from large to small}
Recall the polynomials $P_n^{sm}(q)$, which are equal to $q^{n(n-1)/2}r_n(q)$, $r_n(q)$ being the
$q$-Schr\"oder polynomials. Recall (see \cite{G}) that the polynomials $r_n(q)$ are divisible by $1+q$.
The ratios are denoted by $s_n(q)$ (thus $r_n(q)=s_n(q)(1+q)$). These are the small $q$-Schr\"oder polynomials.
(In particular, $s_n(1)$ are the small Schr\"oder numbers). Our goal here is to show that the divisibility 
of $r_n(q)$ by $1+q$ has a very simple and concrete explanation within our approach.

\begin{thm}
There exists a fixed-point free involution $\sigma$ on the set of smooth collections.
For any smooth collection $\bS$ and the corresponding cell $C_{\bS}$ one has 
\[
\dim C_{\bS} = \dim C_{\sigma\bS} \pm 1.
\]
\end{thm}
\begin{proof}
Consider the map $w:\{1,\dots,n+1\}\to \{1,\dots,n+1\}$, which interchanges $1$ and $n+1$ and stabilizes all other
elements. Define a map $\sigma$ by the formula
\[
\sigma(S_1,\dots,S_n)=(wS_1,\dots,w S_n).
\]
First, we note that $\sigma$ maps each smooth $\bS$ to a smooth collection. Second, since $w^2$ is the identity map,
$\sigma^2=\mathrm{Id}$. Third, let us show that $\sigma$ is fixed-point free. In fact, 
a smooth $\bS$ is fixed by $\sigma$ if and only if  for all $k=1,\dots,n$ the set $S_k$ either contains both
$1$ and $n+1$ or does not contain any of these elements. We note that $\# S_n=n$ and hence $S_n$
contains at least one of the elements $1$, $n+1$. If $\sigma \bS=\bS$, then $S_n\supset \{1,n+1\}$.    
Now let $1\le k<n$ be a number such that $\{1,n+1\}$ is contained in $S_{k+1}$ but not in $S_{k+1}$
(since $\# S_1=1$ such $k$ does exist). If $\sigma\bS=\bS$, then we have $1,n+1\notin S_k$. Since $\bS$ is
smooth, $S_k\subset S_{k+1}$ and therefore $S_{k+1}$ contains two non-intersecting sets $S_k$ and $\{1,n+1\}$.
This contradicts with $\#S_{k+1}=k+1$.

Now let $\bS$ be a smooth collection. Let $k$ be a number such that $1\in S_k\setminus S_{k-1}$ and, 
similarly, let $l$ be a number such that $n+1\in S_l\setminus S_{l-1}$. As we proved above, $k\ne l$.
Assume that $k<l$. We claim that  
\[
\dim C_{\bS} = \dim C_{\sigma\bS} + 1.
\]
Recall that $\dim C_{\bS}$ is the sum of three numbers $N_{PI}(\bS)+N_{PP}(\bS)+N_{II}(\bS)$ 
(see Proposition \ref{3}).
First, we note that a pair $i=1,j=n+1$ adds one to $N_{PI}(\bS)$, but not to $N_{PI}(\sigma\bS)$.
Second, each pair $i,j$ with $1<i,j<n+1$, either shows up for both $\bS$ and $\sigma\bS$ in the
dimension counting as in Proposition \ref{3} or does not show up for both cells. 
Now let us look at other pairs and compute the difference between the dimensions of $C_{\bS}$
and that of $C_{\sigma\bS}$. 

Take $m$ satisfying $k\le m<l$ and consider $j$ such that $j>m$, $j\notin S_m$. Then a pair $i=1,j$ adds one 
to $N_{PI}(\bS)$, 
but not to $N_{PI}(\sigma\bS)$ (since $1\in S_m$, but $1\notin (\sigma S)_m$). However, let us look at a pair
$i=m$, $j=n+1$. Since $n+1\in (\sigma S)_m\setminus S_m$, the pair $(m,n+1)$ adds one to $N_{II}(\sigma\bS)$,
but not to $N_{II}(\bS)$. 

Now take $m$ satisfying $k\le m<l$ and consider $i$ such that $i\le m$, $i\in S_m$. 
Then a pair $i,j=n+1$ adds one 
to $N_{PI}(\bS)$, 
but not to $N_{PI}(\sigma\bS)$ (since $n+1\notin S_m$, but $n+1\in (\sigma S)_m$). However, let us look at a pair
$i=1$, $j=m$. Since $1\in S_m\setminus (\sigma S)_m$, the pair $(1,m)$ adds one to $N_{PP}(\sigma\bS)$,
but not to $N_{PP}(\bS)$. 

Summarizing, the difference    
\[
N_{PI}(\bS)+N_{PP}(\bS)+N_{II}(\bS) - N_{PI}(\sigma\bS) - N_{PP}(\sigma\bS) - N_{II}(\sigma\bS)
\]
is equal to one (coming from the pair $i=1$, $j=n+1$). This implies the second statement of the theorem. 
\end{proof}

\begin{cor}
The polynomials $P_n^{sm}(q)$ and $r_n(q)$ are divisible by $1+q$. The ratio
$P_n^{sm}(q)/(1+q)$ is equal to the sum of the terms $q^{\dim C_{\bS}}$ taken over smooth $\bS$ satisfying
the following conditions for all $m=1,\dots,n$: if $1\in S_m$ then $n+1\in S_m$.
\end{cor}
\begin{proof}
The Theorem above states that $P_n^{sm}(q)$ is equal to the sum over the orbits of the involution $\sigma$
of the terms $q^d(1+q)$, where $d$ is the minimum of the dimensions of the cells corresponding to the collections
in the orbit. But we know that $\dim C_{\bS}= \dim C_{\sigma\bS} - 1$ if there exists $m$ such that
$n+1\in S_m$, but $1\notin S_m$. This implies the corollary.
\end{proof}

Let us relabel the smooth collections as follows. To a smooth collection $\bS$ we attach a permutation
$\pi\in S_{n+1}$ by the formula $\pi(m)=S_m\setminus S_{m-1}$. Then $\bS$ is smooth if and only if the corresponding
permutation satisfies the following conditions for all $1\le a<b\le n$:
\[
\text{if } \pi^{-1}(b)\le a\ \text{  then  }\ \pi^{-1} (a+1)\le b. 
\]

\begin{cor}
The number of permutations, corresponding to smooth collections,  is equal to the large 
Schr\"oder number. The number of such permutations
satisfying $\pi^{-1}(n+1)<\pi^{-1}(1)$ is equal to the small Schr\"oder number.
\end{cor}

\section{Appendix A: Regularity in codimension $2$.}
We consider the Grassmannian  ${\rm Gr}_{\dimv P}(P\oplus I)$ for $P$ a projective and $I$ 
an injective representation over a Dynkin quiver $Q$. Recall that a variety $X$ is said to be regular in codimension
$d$ if there exists a codimension $d+1$ subvariety $Y\subset X$ such that all points of $X\setminus Y$ are
smooth. For example, normal varieties are regular in codimension one. In \cite{CFR} it is
proved that all ${\rm Gr}_{\dimv P}(P\oplus I)$ are normal. We now prove a stronger statement.  
\begin{thm}
${\rm Gr}_{\dimv P}(P\oplus I)$ is regular in codimension $2$. 
\end{thm}
\begin{proof}
Recall that the group 
$\fA\subset {\rm Aut} (P\oplus I)$ acts on ${\rm Gr}_{\dimv P}(P\oplus I)$ with orbits parametrized by pairs of representations $N_I$, $Q_P$ of the same dimension vector such that
$N_I$ is a subrepresentation of $I$ and $Q_P$ is a quotient of $P$.
Assume that an orbit, parametrized by a 
pair $(N_I,Q_P)$  of  dimension vector ${\bf f}$, and admitting exact 
sequences
$$0\rightarrow N_I\rightarrow I\rightarrow Q_I\rightarrow 0,\;\;\; 0\rightarrow N_p\rightarrow P
\rightarrow Q_P\rightarrow 0,$$
is a singular codimension $2$ stratum. Using the codimension formula of the proof of \cite[Theorem 4.5]{CFR}, this means that 
$$\langle {\bf f},{\bf f}\rangle+[N_I,N_I]^1+[Q_P,Q_P]^1=2\qquad \text{ and }\qquad [N_I,Q_P]^1\not=0$$ 
(we use the abbreviations $[X,Y]=\dim{\rm Hom}(X,Y)$ and $[X,Y]^1=\dim{\rm Ext}^1(X,Y)$). If $\langle{\bf f},{\bf f}\rangle=0$, then ${\bf f}=0$, thus $N_I=0=Q_P$, and all 
extension groups are zero, a contradiction. If $\langle{\bf f},{\bf f}\rangle=2$, then $[N_I,N_I]^1=0=[Q_P,Q_P]^1$,
thus both $N_I$ and $Q_P$ are isomorphic to the unique exceptional representation $G$ 
of dimension vector 
${\bf f}$. In particular, $[N_I,Q_P]^1=[G,G]^1=0$, a contradiction. Thus we have 
$\langle {\bf f},{\bf f}\rangle=1$ and (without loss of generality) $[N_I,N_I]^1=0$ and $[Q_P,Q_P]^1=1$. Thus 
${\bf f}$ is a root and $N_I$ is the corresponding indecomposable. $Q_P$ is a minimal degeneration of $N_I$, 
thus by \cite[Theorem 4.5]{B}) there exists a non-split short 
exact sequence
$$0\rightarrow U\rightarrow N_I\rightarrow V\rightarrow 0$$
such that both $U$ and $V$ are indecomposable, and $Q_P\simeq U\oplus V$. In particular, $[V,U]^1\ne 0$, 
thus $[U,V]^1=0$ since Dynkin quivers are representation-directed. We thus have $1=[Q_P,Q_P]^1=[U\oplus V,U\oplus V]^1=[V,U]^1$. 
From $[N_I,N_I]^1=0$ it follows that $[N_I,V]^1=0$ using the above exact sequence, thus $0\not=[N_I,Q_P]^1=[N_I,U\oplus V]^1=[N_I,U]^1$. Applying ${\rm Hom}(\_,U)$ to the above sequence 
yields
$${\rm Hom}(U,U)\rightarrow {\rm Ext}^1(V,U)\rightarrow{\rm Ext}^1(N_I,U)\rightarrow{\rm Ext}^1(U,U)=0.$$
The first two terms in this sequence are both one-dimensional. The connecting map is non-zero since the above 
exact sequence is non-split, thus it is invertible. This implies that $[N_I,U]^1=0$, a contradiction.
\end{proof}

\section{Appendix B: desingularization and the smooth locus.}
Let $\pi_{n+1}:R_{n+1}\to \Fl^a_{n+1}$ be the desingularization of the degenerate flag variety of type $A_n$ of \cite{FF}.   
Our goal here is to prove the following theorem.
\begin{thm}\label{oto}
$\pi^{-1}_{n+1}(x)$ is a single point iff $x$ is a smooth point of $\Fl^a_{n+1}$.
\end{thm}

Recall that $R_{n+1}$ can be explicitly realized as follows. Let $W$ be an $(n+1)$-dimensional
space with a basis $(w_1,\dots,w_{n+1})$. For a pair $1\le i\le j\le n$, let 
$W^{n+1}_{i,j}=\mathrm{span}(w_1,\dots,w_i,w_{j+1},\dots,w_{n+1})$. Then $R_n$ is the variety of collections 
$(V_{i,j})_{1\le i\le j\le n}$ such that $V_{i,j}\in \mathrm{Gr}_i(W^{n+1}_{i,j})$ and
$V_{i,j}\subset V_{i+1,j}$ and $pr_{j+1}V_{i,j}\subset V_{i,j+1}$. 

\begin{lem}
$R_{n+1}$ can be embedded into $\Fl^a_{n+1}\times R_n$ in such a way that $\pi_{n+1}$ is simply the projection 
to the first factor.
\end{lem} 
\begin{proof}
We first note that the map $\pi_{n+1}:R_{n+1}\to \Fl^a_{n+1}$ is explicitly given by 
$(V_{i,j})_{i\le j}\mapsto (V_{i,i})_{i=1}^n$. Now consider the  forgetful map 
$$(V_{i,j})_{1\le i\le j\le n}\to (V_{i,j})_{1\le i< j\le n}$$ (the diagonal terms $V_{i,i}$ are omitted).
We claim that the image is isomorphic to $R_n$.  Namely, for a pair $1\le i<j\le n$, we consider the 
"shift" map
$sh_{i,j}: W_{i,j}^{n+1}\to W_{i,j-1}^n$ given by
\[
sh_{i,j} w_k=\begin{cases} w_k, \text{ if } k\le i,\\ w_{k-1}, \text{ if } k>j.\end{cases}
\]
Then for a point $(V_{i,j})_{i\le j}\in R_{n+1}$, the collection 
\[
(V'_{i,j})_{1\le i\le j\le n-1}=(sh_{i,j+1} V_{i,j+1})_{1\le i\le j\le n-1} 
\]
belongs to $R_n$. We denote the map $R_{n+1}\to R_n$ by $\psi_{n+1}$. Now the embedding
$R_{n+1}\to \Fl^a_{n+1}\times R_n$ is given by the map $A=(\pi_{n+1},\psi_{n+1})$.
\end{proof}

\begin{lem}
Let $\bS$ be a length $n$ smooth collection. Then 
\[
\pi_n\psi_{n+1} \pi_{n+1}^{-1} p_{\bS}\subset\Fl^a_n
\]
is a single point.
Moreover, it is a smooth torus fixed point.  
\end{lem}
\begin{proof}
Recall that 
\[
p_{\bS}=((p_{\bS})_i)_{i=1}^n,\ (p_{\bS})_i=\mathrm{span}(w_a:\ a\in S_i).
\]
Our first goal is to prove that there exists a unique way to define spaces $(V_{i,i+1})_{i=1}^{n-1}$ such
that there exists a point in $R_{n+1}$ with the diagonal components being $(p_{\bS})_i$ and the $(i,i+1)$-st
components being $V_{i,i+1}$. In fact, fix some $i$ with $1\le i\le n-1$. We need $V_{i,i+1}$ such that
$\dim V_{i,i+1}=i$ and 
\[
pr_{i+1} (p_{\bS})_i\subset V_{i,i+1}\subset W^{n+1}_{i,i+1}\cap (p_{\bS})_{i+1}.
\] 
If $i+1\notin S_i$, then $\dim pr_{i+1} (p_{\bS})_i=i$ and hence $V_{i,i+1}=pr_{i+1} (p_{\bS})_i$. If 
$i+1\in S_i$, then since $\bS$ is smooth, we have $i+1\in S_{i+1}$. Therefore the intersection 
\[
W^{n+1}_{i,i+1}\cap (p_{\bS})_{i+1}=\mathrm{span}(w_a: a\ne i+1)\cap \mathrm{span}(w_a:\ a\in S_{i+1})
\] 
is $i$-dimensional and hence $V_{i,i+1}$ is forced to coincide with this intersection.
Note that in both cases $V_{i,i+1}$ is the linear span of some basis vectors. We denote by 
$S_{i,i+1}\subset \{1,\dots,i,i+2,\dots,n+1\}$ the set of indices of these vectors, i.e.
\[
V_{i,i+1}=\mathrm{span}(w_a:\ a\in S_{i,i+1}).
\]
We note that $S_{i,i+1}\subset S_{i+1}$ and $S_i\subset S_{i,i+1}\cup\{i+1\}$.
 
We identify the collection of subspaces $(V_{i,i+1})_{i=1}^{n-1}$ constructed above with the point 
$(sh_{i,i+1}V_{i,i+1})_{i=1}^{n-1}\in\Fl^a_n$. As mentioned above, each component of this point
is a linear span of basis vectors and thus $(sh_{i,i+1}V_{i,i+1})_{i=1}^{n-1}=p(\bar \bS)$
for some collection $\bar\bS=(\bar S_1,\dots,\bar S_{n-1})$. 
Explicitly,
\[
\bar S_i=\{a:\ a\in S_{i,i+1}, a\le i\}\cup \{a-1:\ a\in S_{i,i+1}, a> i+1\}.
\]
Our goal is to prove that this collection is
smooth. In fact, assume $b\in \bar S_a$ for some $1\le a<b\le n-1$. Then since $b>a$ we have
$b+1\in S_{a,a+1}$. We consider two cases: $b+1\in S_a$ and $b+1\notin S_a$. If $b+1\in S_a$, 
then $a+1\in S_{b+1}$ ($\bS$ is smooth). Since $S_a\subset S_{b+1}$, we have $b+1\in S_{b+1}$. Therefore,
$S_{b,b+1}=S_{b+1}\setminus \{b+1\}$ and, in particular, $a+1\in S_{b,b+1}$. Since $a+1\le b$, this implies 
$a+1\in \bar S_b$. Now assume $b+1\notin S_a$. Then $S_{a,a+1}\ne S_a$ and hence $a+1\in S_a$. This
implies $a+1\in S_b$ and so $a+1\in S_{b,b+1}$ (because $w_{a+1}=pr_{b+1} w_{a+1}\in V_{b,b+1}$).
We thus arrive at $a+1\in\bar S_b$, which means that $\bar \bS$ is smooth.   
\end{proof}

\begin{cor}
The map $\pi_{n+1}$ is one-to-one over the smooth locus of $\Fl^a_{n+1}$. 
\end{cor}
\begin{proof}
We note that since the fibers over any two points of a given cell in $\Fl^a_{n+1}$ are isomorphic,
it suffices to prove that the fiber is a single point over a smooth torus fixed point.
Let $\bS$ be a smooth collection and $p(\bar\bS)=\pi_n\psi_{n+1} \pi_{n+1}^{-1}$.  
Since $\bar\bS$ is smooth, our corollary
follows by induction on $n$. 
\end{proof}

To complete the proof of Theorem \ref{oto}, we need to show that the fiber over a non-smooth point
has positive dimension. It suffices to prove that if a collection $\bS$ is not smooth, then the
preimage of $p_{\bS}$ has positive dimension. 
We first prove the following lemma.
\begin{lem}
Assume that $S_a$ is not a subset of $S_{a+1}$ for some $a$. Then the dimension of the fiber
$\pi_{n+1}^{-1} p_{\bS}$ is positive.
\end{lem}
\begin{proof}
Assume that $p_{\bS}$ is the image of $(V_{i,j})_{1\le i\le j\le n}$. Let us look at possible sets
$V_{a,a+1}$. We know that 
\begin{equation}\label{ii}
pr_{a+1} (p_{\bS})_a\subset V_{a,a+1}\subset (p_{\bS})_{a+1}\cap \mathrm{span}(w_i: i\ne a+1).
\end{equation}
Since $S_a$ is not a subset of $S_{a+1}$ and $S_a\subset S_{a+1}\cup\{a+1\}$, we obtain $a+1\in S_a$,
$a+1\notin S_{a+1}$. Therefore,
\[
\dim pr_{a+1} (p_{\bS})_a=a-1,\ \dim (p_{\bS})_{a+1}\cap \mathrm{span}(w_i: i\ne a+1)=a+1.
\]
Thus the choice of $V_{i,i+1}$ as in \eqref{ii} is equivalent to the choice of a point in $\bP^1$.
Therefore the preimage  $\pi^{-1}_{n+1} p_{\bS}$ is at least one-dimensional. 
\end{proof}  

\begin{cor}
If $\bS$ is not smooth, then the dimension of the fiber $\pi^{-1}_{n+1} p_{\bS}$ is positive.
\end{cor}
\begin{proof}
Let $k\ge 1$ be a minimal number such that there exists a number $a$, $1\le a\le n-k$ such that
$a+k\in S_a$, but $a+1\notin S_{a+k}$. We prove our corollary by induction on $k$. First, we
note that the case $k=1$ means that $S_a\notin S_{a+1}$ and we are done by the lemma above.
Now let $k>1$. Since $k>1$ the sets $S_{a,a+1}$ satisfying 
\[
S_a\cup \{a+1\}\subset S_{a,a+1}\subset S_{a+1} 
\]  
are defined uniquely. Now define a length $n-1$ collection  $\bar \bS$ as above:
\[
\bar S_i=\{l:\ l\in S_{i,i+1}, l\le i\}\cup \{l-1:\ l\in S_{i,i+1}, l> i+1\}.
\]
Since $a+k\in S_a$ and $k>1$ we obtain $a+k-1\in \bar S_a$. Also, since $a+1\notin S_{a+k}$, we
obtain $a+1\notin \bar S_{a+k}$ and hence $a+1\notin \bar S_{a+k-1}$ (since $k>1$ we have 
$S_{a+k-1}\subset S_{a+k}$). This proves that $k$ becomes $k-1$ for $\bar\bS$. By the inductive
assumption we know that the preimage $\pi^{-1}_n p(\bar\bS)$ is positive-dimensional.
But $\pi^{-1}_{n+1} p_{\bS}=\pi^{-1}_n p(\bar\bS)$ and we are done. 
\end{proof}

\section{Appendix C}
In this appendix we compute the moment graph of $\Fl_{4}^a$.
The T--fixed points of $\Fl_4^a$ are listed in figure \ref{Fig:TFixedA3}.  Recall that such points are parameterized by successor--closed subquivers of the following quiver
\begin{equation}\label{Eq:DiagBasis}
\xymatrix@R=3pt@C=10pt{3&&&\cdot\\2&&\cdot\ar[r]&\cdot\\1&\cdot\ar[r]&\cdot\ar[r]&\cdot\\4&\cdot\ar[r]&\cdot\ar[r]&\cdot\\3&\cdot\ar[r]&\cdot&\\2&\cdot&&}
\end{equation}
having one vertex in the first column, two  in the second and  three vertices in the third column.
 \begin{figure}
$$
\begin{array}{|c|c|c|c|c|c|c|}
\hline
1\xymatrix@R=3pt@C=3pt{&&*[F]{\cdot}\\&*[F]{\cdot}\ar@{-}[r]&*[F]{\cdot}\\*[F]{\cdot}\ar@{-}[r]&*[F]{\cdot}\ar@{-}[r]&*[F]{\cdot}\\*{\cdot}\ar@{-}[r]&*{\cdot}\ar@{-}[r]&*{\cdot}\\*{\cdot}\ar@{-}[r]&*{\cdot}&\\*{\cdot}&&}&
2\xymatrix@R=3pt@C=3pt{&&*[F]{\cdot}\\&*[F]{\cdot}\ar@{-}[r]&*[F]{\cdot}\\*{\cdot}\ar@{-}[r]&*{\cdot}\ar@{-}[r]&*{\cdot}\\*[F]{\cdot}\ar@{-}[r]&*[F]{\cdot}\ar@{-}[r]&*[F]{\cdot}\\*{\cdot}\ar@{-}[r]&*{\cdot}&\\*{\cdot}&&}&
3\xymatrix@R=3pt@C=3pt{&&*[F]{\cdot}\\&*[F]{\cdot}\ar@{-}[r]&*[F]{\cdot}\\*{\cdot}\ar@{-}[r]&*{\cdot}\ar@{-}[r]&*[F]{\cdot}\\*{\cdot}\ar@{-}[r]&*{\cdot}\ar@{-}[r]&*{\cdot}\\*[F]{\cdot}\ar@{-}[r]&*[F]{\cdot}&\\*{\cdot}&&}&
4\xymatrix@R=3pt@C=3pt{&&*[F]{\cdot}\\&*[F]{\cdot}\ar@{-}[r]&*[F]{\cdot}\\*{\cdot}\ar@{-}[r]&*[F]{\cdot}\ar@{-}[r]&*[F]{\cdot}\\*{\cdot}\ar@{-}[r]&*{\cdot}\ar@{-}[r]&*{\cdot}\\*{\cdot}\ar@{-}[r]&*{\cdot}&\\*[F]{\cdot}&&}&
5\xymatrix@R=3pt@C=3pt{&&*[F]{\cdot}\\&*{\cdot}\ar@{-}[r]&*{\cdot}\\*[F]{\cdot}\ar@{-}[r]&*[F]{\cdot}\ar@{-}[r]&*[F]{\cdot}\\*{\cdot}\ar@{-}[r]&*[F]{\cdot}\ar@{-}[r]&*[F]{\cdot}\\*{\cdot}\ar@{-}[r]&*{\cdot}&\\*{\cdot}&&}& 
6\xymatrix@R=3pt@C=3pt{&&*[F]{\cdot}\\&*{\cdot}\ar@{-}[r]&*[F]{\cdot}\\*[F]{\cdot}\ar@{-}[r]&*[F]{\cdot}\ar@{-}[r]&*[F]{\cdot}\\*{\cdot}\ar@{-}[r]&*{\cdot}\ar@{-}[r]&*{\cdot}\\*{\cdot}\ar@{-}[r]&*[F]{\cdot}&\\*{\cdot}&&}&
7\xymatrix@R=3pt@C=3pt{&&*{\cdot}\\&*[F]{\cdot}\ar@{-}[r]&*[F]{\cdot}\\*[F]{\cdot}\ar@{-}[r]&*[F]{\cdot}\ar@{-}[r]&*[F]{\cdot}\\*{\cdot}\ar@{-}[r]&*{\cdot}\ar@{-}[r]&*[F]{\cdot}\\*{\cdot}\ar@{-}[r]&*{\cdot}&\\*{\cdot}&&}\\\hline
\hline
8\xymatrix@R=3pt@C=3pt{&&*[F]{\cdot}\\&*[F]{\cdot}\ar@{-}[r]&*[F]{\cdot}\\*{\cdot}\ar@{-}[r]&*{\cdot}\ar@{-}[r]&*{\cdot}\\*{\cdot}\ar@{-}[r]&*{\cdot}\ar@{-}[r]&*[F]{\cdot}\\*[F]{\cdot}\ar@{-}[r]&*[F]{\cdot}&\\*{\cdot}&&}&
9\xymatrix@R=3pt@C=3pt{&&*[F]{\cdot}\\&*[F]{\cdot}\ar@{-}[r]&*[F]{\cdot}\\*{\cdot}\ar@{-}[r]&*{\cdot}\ar@{-}[r]&*{\cdot}\\*{\cdot}\ar@{-}[r]&*[F]{\cdot}\ar@{-}[r]&*[F]{\cdot}\\*{\cdot}\ar@{-}[r]&*{\cdot}&\\*[F]{\cdot}&&}&
10\xymatrix@R=3pt@C=3pt{&&*[F]{\cdot}\\&*{\cdot}\ar@{-}[r]&*{\cdot}\\*{\cdot}\ar@{-}[r]&*[F]{\cdot}\ar@{-}[r]&*[F]{\cdot}\\*[F]{\cdot}\ar@{-}[r]&*[F]{\cdot}\ar@{-}[r]&*[F]{\cdot}\\*{\cdot}\ar@{-}[r]&*{\cdot}&\\*{\cdot}&&}&
11\xymatrix@R=3pt@C=3pt{&&*{\cdot}\\&*[F]{\cdot}\ar@{-}[r]&*[F]{\cdot}\\*{\cdot}\ar@{-}[r]&*{\cdot}\ar@{-}[r]&*[F]{\cdot}\\*[F]{\cdot}\ar@{-}[r]&*[F]{\cdot}\ar@{-}[r]&*[F]{\cdot}\\*{\cdot}\ar@{-}[r]&*{\cdot}&\\*{\cdot}&&}&
12\xymatrix@R=3pt@C=3pt{&&*[F]{\cdot}\\&*{\cdot}\ar@{-}[r]&*[F]{\cdot}\\*{\cdot}\ar@{-}[r]&*{\cdot}\ar@{-}[r]&*{\cdot}\\*[F]{\cdot}\ar@{-}[r]&*[F]{\cdot}\ar@{-}[r]&*[F]{\cdot}\\*{\cdot}\ar@{-}[r]&*[F]{\cdot}&\\*{\cdot}&&}& 
13\xymatrix@R=3pt@C=3pt{&&*[F]{\cdot}\\&*[F]{\cdot}\ar@{-}[r]&*[F]{\cdot}\\*{\cdot}\ar@{-}[r]&*{\cdot}\ar@{-}[r]&*[F]{\cdot}\\*{\cdot}\ar@{-}[r]&*{\cdot}\ar@{-}[r]&*{\cdot}\\*{\cdot}\ar@{-}[r]&*[F]{\cdot}&\\*[F]{\cdot}&&}&
14\xymatrix@R=3pt@C=3pt{&&*[F]{\cdot}\\&*{\cdot}\ar@{-}[r]&*[F]{\cdot}\\*{\cdot}\ar@{-}[r]&*[F]{\cdot}\ar@{-}[r]&*[F]{\cdot}\\*{\cdot}\ar@{-}[r]&*{\cdot}\ar@{-}[r]&*{\cdot}\\*[F]{\cdot}\ar@{-}[r]&*[F]{\cdot}&\\*{\cdot}&&}\\\hline
\hline
15\xymatrix@R=3pt@C=3pt{&&*[F]{\cdot}\\&*{\cdot}\ar@{-}[r]&*{\cdot}\\*{\cdot}\ar@{-}[r]&*{\cdot}\ar@{-}[r]&*[F]{\cdot}\\*{\cdot}\ar@{-}[r]&*[F]{\cdot}\ar@{-}[r]&*[F]{\cdot}\\*[F]{\cdot}\ar@{-}[r]&*[F]{\cdot}&\\*{\cdot}&&}&
16\xymatrix@R=3pt@C=3pt{&&*{\cdot}\\&*[F]{\cdot}\ar@{-}[r]&*[F]{\cdot}\\*{\cdot}\ar@{-}[r]&*{\cdot}\ar@{-}[r]&*[F]{\cdot}\\*{\cdot}\ar@{-}[r]&*{\cdot}\ar@{-}[r]&*[F]{\cdot}\\*[F]{\cdot}\ar@{-}[r]&*[F]{\cdot}&\\*{\cdot}&&}&
17\xymatrix@R=3pt@C=3pt{&&*[F]{\cdot}\\&*{\cdot}\ar@{-}[r]&*[F]{\cdot}\\*{\cdot}\ar@{-}[r]&*[F]{\cdot}\ar@{-}[r]&*[F]{\cdot}\\*{\cdot}\ar@{-}[r]&*{\cdot}\ar@{-}[r]&*{\cdot}\\*{\cdot}\ar@{-}[r]&*[F]{\cdot}&\\*[F]{\cdot}&&}&
18\xymatrix@R=3pt@C=3pt{&&*[F]{\cdot}\\&*{\cdot}\ar@{-}[r]&*{\cdot}\\*{\cdot}\ar@{-}[r]&*[F]{\cdot}\ar@{-}[r]&*[F]{\cdot}\\*{\cdot}\ar@{-}[r]&*[F]{\cdot}\ar@{-}[r]&*[F]{\cdot}\\*{\cdot}\ar@{-}[r]&*{\cdot}&\\*[F]{\cdot}&&}&
19\xymatrix@R=3pt@C=3pt{&&*{\cdot}\\&*[F]{\cdot}\ar@{-}[r]&*[F]{\cdot}\\*{\cdot}\ar@{-}[r]&*[F]{\cdot}\ar@{-}[r]&*[F]{\cdot}\\*{\cdot}\ar@{-}[r]&*{\cdot}\ar@{-}[r]&*[F]{\cdot}\\*{\cdot}\ar@{-}[r]&*{\cdot}&\\*[F]{\cdot}&&}& 
20\xymatrix@R=3pt@C=3pt{&&*[F]{\cdot}\\&*{\cdot}\ar@{-}[r]&*{\cdot}\\*[F]{\cdot}\ar@{-}[r]&*[F]{\cdot}\ar@{-}[r]&*[F]{\cdot}\\*{\cdot}\ar@{-}[r]&*{\cdot}\ar@{-}[r]&*[F]{\cdot}\\*{\cdot}\ar@{-}[r]&*[F]{\cdot}&\\*{\cdot}&&}&
21\xymatrix@R=3pt@C=3pt{&&*{\cdot}\\&*{\cdot}\ar@{-}[r]&*[F]{\cdot}\\*[F]{\cdot}\ar@{-}[r]&*[F]{\cdot}\ar@{-}[r]&*[F]{\cdot}\\*{\cdot}\ar@{-}[r]&*[F]{\cdot}\ar@{-}[r]&*[F]{\cdot}\\*{\cdot}\ar@{-}[r]&*{\cdot}&\\*{\cdot}&&}\\\hline
\hline
22\xymatrix@R=3pt@C=3pt{&&*{\cdot}\\&*{\cdot}\ar@{-}[r]&*[F]{\cdot}\\*[F]{\cdot}\ar@{-}[r]&*[F]{\cdot}\ar@{-}[r]&*[F]{\cdot}\\*{\cdot}\ar@{-}[r]&*{\cdot}\ar@{-}[r]&*[F]{\cdot}\\*{\cdot}\ar@{-}[r]&*[F]{\cdot}&\\*{\cdot}&&}&
23\xymatrix@R=3pt@C=3pt{&&*[F]{\cdot}\\&*[F]{\cdot}\ar@{-}[r]&*[F]{\cdot}\\*{\cdot}\ar@{-}[r]&*{\cdot}\ar@{-}[r]&*{\cdot}\\*{\cdot}\ar@{-}[r]&*{\cdot}\ar@{-}[r]&*[F]{\cdot}\\*{\cdot}\ar@{-}[r]&*[F]{\cdot}&\\*[F]{\cdot}&&}&
24\xymatrix@R=3pt@C=3pt{&&*[F]{\cdot}\\&*{\cdot}\ar@{-}[r]&*[F]{\cdot}\\*{\cdot}\ar@{-}[r]&*{\cdot}\ar@{-}[r]&*{\cdot}\\*{\cdot}\ar@{-}[r]&*[F]{\cdot}\ar@{-}[r]&*[F]{\cdot}\\*[F]{\cdot}\ar@{-}[r]&*[F]{\cdot}&\\*{\cdot}&&}&
25\xymatrix@R=3pt@C=3pt{&&*[F]{\cdot}\\&*{\cdot}\ar@{-}[r]&*{\cdot}\\*{\cdot}\ar@{-}[r]&*[F]{\cdot}\ar@{-}[r]&*[F]{\cdot}\\*{\cdot}\ar@{-}[r]&*{\cdot}\ar@{-}[r]&*[F]{\cdot}\\*[F]{\cdot}\ar@{-}[r]&*[F]{\cdot}&\\*{\cdot}&&}&
26\xymatrix@R=3pt@C=3pt{&&*[F]{\cdot}\\&*{\cdot}\ar@{-}[r]&*[F]{\cdot}\\*{\cdot}\ar@{-}[r]&*{\cdot}\ar@{-}[r]&*{\cdot}\\*{\cdot}\ar@{-}[r]&*[F]{\cdot}\ar@{-}[r]&*[F]{\cdot}\\*{\cdot}\ar@{-}[r]&*[F]{\cdot}&\\*[F]{\cdot}&&}& 
27\xymatrix@R=3pt@C=3pt{&&*{\cdot}\\&*[F]{\cdot}\ar@{-}[r]&*[F]{\cdot}\\*{\cdot}\ar@{-}[r]&*{\cdot}\ar@{-}[r]&*[F]{\cdot}\\*{\cdot}\ar@{-}[r]&*[F]{\cdot}\ar@{-}[r]&*[F]{\cdot}\\*{\cdot}\ar@{-}[r]&*{\cdot}&\\*[F]{\cdot}&&}&
28\xymatrix@R=3pt@C=3pt{&&*[F]{\cdot}\\&*{\cdot}\ar@{-}[r]&*{\cdot}\\*{\cdot}\ar@{-}[r]&*{\cdot}\ar@{-}[r]&*[F]{\cdot}\\*[F]{\cdot}\ar@{-}[r]&*[F]{\cdot}\ar@{-}[r]&*[F]{\cdot}\\*{\cdot}\ar@{-}[r]&*[F]{\cdot}&\\*{\cdot}&&}\\\hline
\hline
29\xymatrix@R=3pt@C=3pt{&&*{\cdot}\\&*{\cdot}\ar@{-}[r]&*[F]{\cdot}\\*{\cdot}\ar@{-}[r]&*[F]{\cdot}\ar@{-}[r]&*[F]{\cdot}\\*[F]{\cdot}\ar@{-}[r]&*[F]{\cdot}\ar@{-}[r]&*[F]{\cdot}\\*{\cdot}\ar@{-}[r]&*{\cdot}&\\*{\cdot}&&}&
30\xymatrix@R=3pt@C=3pt{&&*{\cdot}\\&*{\cdot}\ar@{-}[r]&*[F]{\cdot}\\*{\cdot}\ar@{-}[r]&*{\cdot}\ar@{-}[r]&*[F]{\cdot}\\*[F]{\cdot}\ar@{-}[r]&*[F]{\cdot}\ar@{-}[r]&*[F]{\cdot}\\*{\cdot}\ar@{-}[r]&*[F]{\cdot}&\\*{\cdot}&&}&
31\xymatrix@R=3pt@C=3pt{&&*[F]{\cdot}\\&*{\cdot}\ar@{-}[r]&*{\cdot}\\*{\cdot}\ar@{-}[r]&*{\cdot}\ar@{-}[r]&*[F]{\cdot}\\*{\cdot}\ar@{-}[r]&*[F]{\cdot}\ar@{-}[r]&*[F]{\cdot}\\*{\cdot}\ar@{-}[r]&*[F]{\cdot}&\\*[F]{\cdot}&&}&
32\xymatrix@R=3pt@C=3pt{&&*{\cdot}\\&*[F]{\cdot}\ar@{-}[r]&*[F]{\cdot}\\*{\cdot}\ar@{-}[r]&*{\cdot}\ar@{-}[r]&*[F]{\cdot}\\*{\cdot}\ar@{-}[r]&*{\cdot}\ar@{-}[r]&*[F]{\cdot}\\*{\cdot}\ar@{-}[r]&*[F]{\cdot}&\\*[F]{\cdot}&&}&
33\xymatrix@R=3pt@C=3pt{&&*{\cdot}\\&*{\cdot}\ar@{-}[r]&*[F]{\cdot}\\*{\cdot}\ar@{-}[r]&*[F]{\cdot}\ar@{-}[r]&*[F]{\cdot}\\*{\cdot}\ar@{-}[r]&*{\cdot}\ar@{-}[r]&*[F]{\cdot}\\*[F]{\cdot}\ar@{-}[r]&*[F]{\cdot}&\\*{\cdot}&&}& 
34\xymatrix@R=3pt@C=3pt{&&*{\cdot}\\&*{\cdot}\ar@{-}[r]&*[F]{\cdot}\\*{\cdot}\ar@{-}[r]&*{\cdot}\ar@{-}[r]&*[F]{\cdot}\\*{\cdot}\ar@{-}[r]&*[F]{\cdot}\ar@{-}[r]&*[F]{\cdot}\\*[F]{\cdot}\ar@{-}[r]&*[F]{\cdot}&\\*{\cdot}&&}&
35\xymatrix@R=3pt@C=3pt{&&*{\cdot}\\&*{\cdot}\ar@{-}[r]&*[F]{\cdot}\\*{\cdot}\ar@{-}[r]&*[F]{\cdot}\ar@{-}[r]&*[F]{\cdot}\\*{\cdot}\ar@{-}[r]&*{\cdot}\ar@{-}[r]&*[F]{\cdot}\\*{\cdot}\ar@{-}[r]&*[F]{\cdot}&\\*[F]{\cdot}&&}\\\hline
\end{array}
$$
$$
\begin{array}{|c|c|c|}
\hline
36\xymatrix@R=3pt@C=3pt{&&*[F]{\cdot}\\&*{\cdot}\ar@{-}[r]&*{\cdot}\\*{\cdot}\ar@{-}[r]&*[F]{\cdot}\ar@{-}[r]&*[F]{\cdot}\\*{\cdot}\ar@{-}[r]&*{\cdot}\ar@{-}[r]&*[F]{\cdot}\\*{\cdot}\ar@{-}[r]&*[F]{\cdot}&\\*[F]{\cdot}&&}&
37\xymatrix@R=3pt@C=3pt{&&*{\cdot}\\&*{\cdot}\ar@{-}[r]&*[F]{\cdot}\\*{\cdot}\ar@{-}[r]&*[F]{\cdot}\ar@{-}[r]&*[F]{\cdot}\\*{\cdot}\ar@{-}[r]&*[F]{\cdot}\ar@{-}[r]&*[F]{\cdot}\\*{\cdot}\ar@{-}[r]&*{\cdot}&\\*[F]{\cdot}&&}&
38\xymatrix@R=3pt@C=3pt{&&*{\cdot}\\&*{\cdot}\ar@{-}[r]&*[F]{\cdot}\\*{\cdot}\ar@{-}[r]&*{\cdot}\ar@{-}[r]&*[F]{\cdot}\\*{\cdot}\ar@{-}[r]&*[F]{\cdot}\ar@{-}[r]&*[F]{\cdot}\\*{\cdot}\ar@{-}[r]&*[F]{\cdot}&\\*[F]{\cdot}&&}\\
\hline
\end{array}
$$
\caption{The $T$--fixed points of $\mathcal{F}_4^a$.}\label{Fig:TFixedA3}
\end{figure}
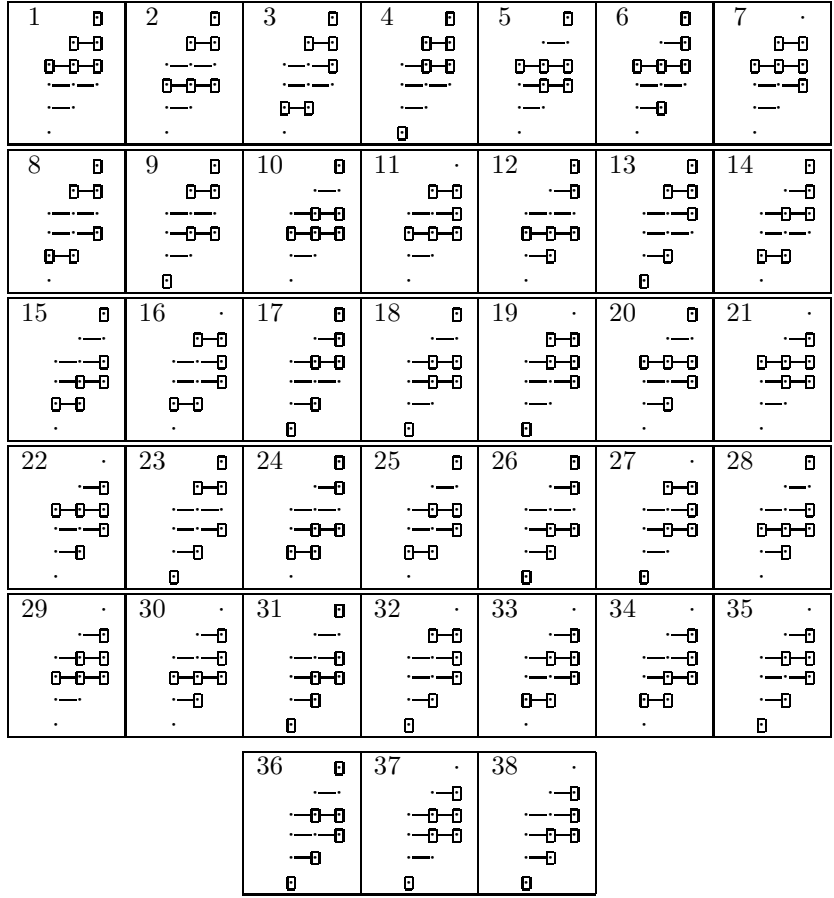

Figure~\ref{Fig:MomGrapgA3} shows the moment graph of the degenerate flag variety $\Fl_4^a$ (We used Bernhard Keller's quiver mutation applet to draw the picture \cite{quiver_mutation}). The 22 smooth torus fixed points are highlighted by a frame. These are the vertices adjacent to precisely $6={\rm dim }\Fl_4^a$ edges. 
An edge $p_\bS$--$p_\mathbf{R}$ of the moment graph corresponds to a $T$--fixed curve between $p_\bS$ and $p_\mathbf{R}$ in $\Fl_4^a$ whose direction around $p_\bS$ and $p_\mathbf{R}$ is given by a standard basis vector of the tangent space at them. The edge is oriented $p_\bS\rightarrow p_\mathbf{R}$ if and only if  the direction around $p_\bS$ has positive $T_0$--weight and it is labelled by the corresponding $\bS$--effective pair (see theorem~\ref{Thm:OneDimOrbits} and remark~\ref{Remark:S-effectiveTorus}). 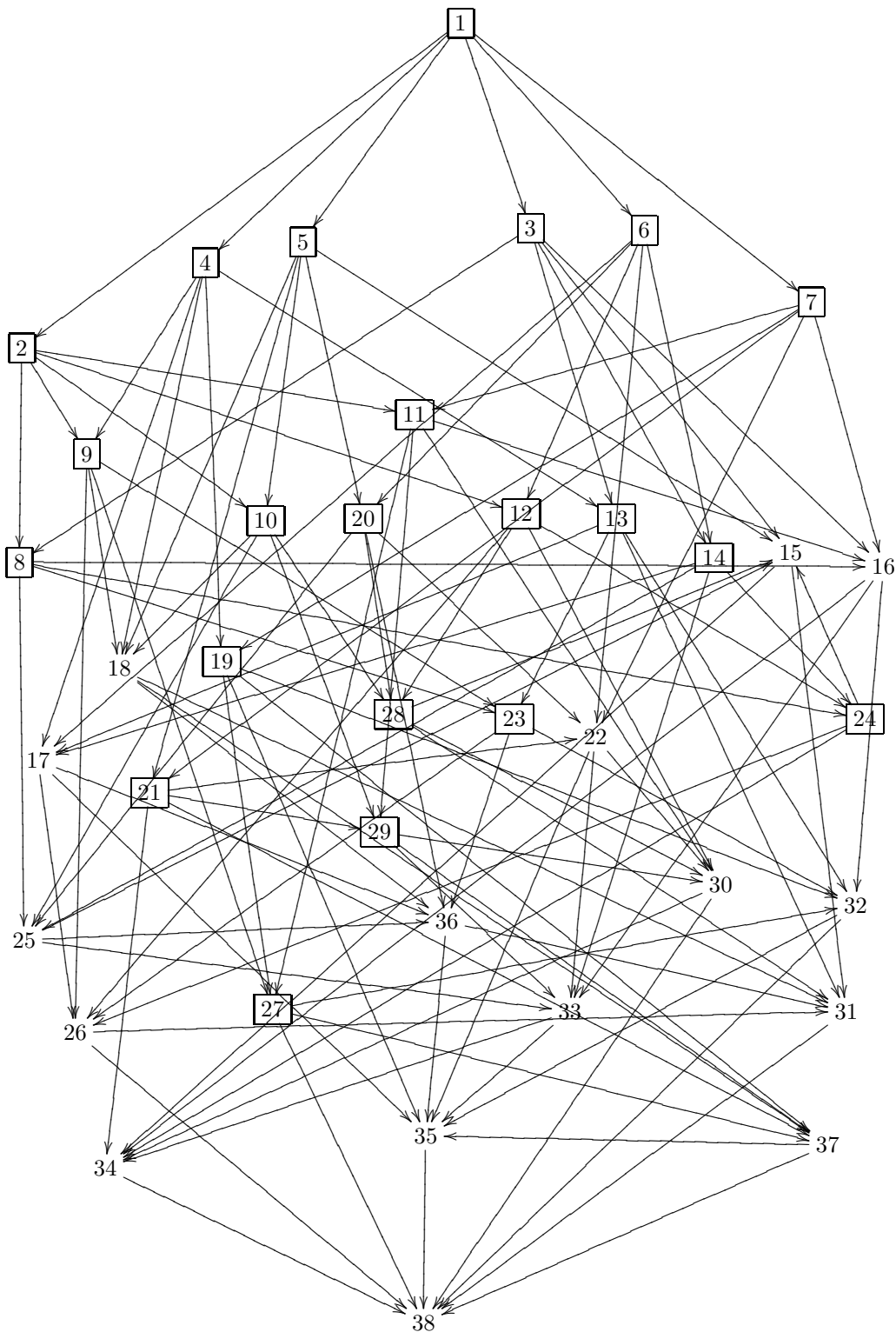
\begin{figure}
\begin{xy} 0;<1pt,0pt>:<0pt,-1pt>:: 
(190,0) *+[F]{1} ="0",
(1,141) *+[F]{2} ="1",
(220,89) *+[F]{3} ="2",
(80,104) *+[F]{4} ="3",
(122,95) *+[F]{5} ="4",
(269,90) *+[F]{6} ="5",
(341,121) *+[F]{7} ="6",
(0,234) *+[F]{8} ="7",
(29,187) *+[F]{9} ="8",
(106,216) *+[F]{10} ="9",
(170,170) *+[F]{11} ="10",
(216,213) *+[F]{12} ="11",
(257,215) *+[F]{13} ="12",
(299,232) *+[F]{14} ="13",
(332,230) *+{15} ="14",
(372,236) *+{16} ="15",
(8,320) *+{17} ="16",
(43,280) *+{18} ="17",
(87,277) *+[F]{19} ="18",
(148,215) *+[F]{20} ="19",
(56,334) *+[F]{21} ="20",
(248,310) *+{22} ="21",
(213,302) *+[F]{23} ="22",
(364,302) *+[F]{24} ="23",
(2,398) *+{25} ="24",
(24,438) *+{26} ="25",
(109,428) *+[F]{27} ="26",
(161,300) *+[F]{28} ="27",
(155,351) *+[F]{29} ="28",
(302,374) *+{30} ="29",
(356,429) *+{31} ="30",
(360,383) *+{32} ="31",
(237,429) *+{33} ="32",
(37,497) *+{34} ="33",
(175,483) *+{35} ="34",
(184,390) *+{36} ="35",
(348,487) *+{37} ="36",
(174,564) *+{38} ="37",
"0", {\ar"1"},
"0", {\ar"2"},
"0", {\ar"3"},
"0", {\ar"4"},
"0", {\ar"5"},
"0", {\ar"6"},
"1", {\ar"7"},
"1", {\ar"8"},
"1", {\ar"9"},
"1", {\ar"10"},
"1", {\ar"11"},
"2", {\ar"7"},
"2", {\ar"12"},
"2", {\ar"13"},
"2", {\ar"14"},
"2", {\ar"15"},
"3", {\ar"8"},
"3", {\ar"12"},
"3", {\ar"16"},
"3", {\ar"17"},
"3", {\ar"18"},
"4", {\ar"9"},
"4", {\ar"14"},
"4", {\ar"17"},
"4", {\ar"19"},
"4", {\ar"20"},
"5", {\ar"11"},
"5", {\ar"13"},
"5", {\ar"16"},
"5", {\ar"19"},
"5", {\ar"21"},
"6", {\ar"10"},
"6", {\ar"15"},
"6", {\ar"18"},
"6", {\ar"20"},
"6", {\ar"21"},
"7", {\ar"15"},
"7", {\ar"22"},
"7", {\ar"23"},
"7", {\ar"24"},
"8", {\ar"17"},
"8", {\ar"22"},
"8", {\ar"25"},
"8", {\ar"26"},
"9", {\ar"17"},
"9", {\ar"24"},
"9", {\ar"27"},
"9", {\ar"28"},
"10", {\ar"15"},
"10", {\ar"26"},
"10", {\ar"28"},
"10", {\ar"29"},
"11", {\ar"23"},
"11", {\ar"25"},
"11", {\ar"27"},
"11", {\ar"29"},
"12", {\ar"16"},
"12", {\ar"22"},
"12", {\ar"30"},
"12", {\ar"31"},
"13", {\ar"16"},
"13", {\ar"23"},
"13", {\ar"24"},
"13", {\ar"32"},
"23", {\ar"14"},
"24", {\ar"14"},
"27", {\ar"14"},
"14", {\ar"30"},
"14", {\ar"33"},
"15", {\ar"31"},
"15", {\ar"32"},
"15", {\ar"33"},
"16", {\ar"25"},
"16", {\ar"34"},
"16", {\ar"35"},
"17", {\ar"30"},
"17", {\ar"35"},
"17", {\ar"36"},
"18", {\ar"26"},
"18", {\ar"31"},
"18", {\ar"34"},
"18", {\ar"36"},
"19", {\ar"21"},
"19", {\ar"24"},
"19", {\ar"27"},
"19", {\ar"35"},
"20", {\ar"21"},
"20", {\ar"28"},
"20", {\ar"33"},
"20", {\ar"36"},
"21", {\ar"29"},
"21", {\ar"32"},
"21", {\ar"34"},
"22", {\ar"25"},
"22", {\ar"31"},
"22", {\ar"35"},
"23", {\ar"25"},
"23", {\ar"33"},
"24", {\ar"32"},
"24", {\ar"35"},
"25", {\ar"30"},
"25", {\ar"37"},
"26", {\ar"31"},
"26", {\ar"36"},
"26", {\ar"37"},
"27", {\ar"29"},
"27", {\ar"30"},
"28", {\ar"29"},
"28", {\ar"32"},
"28", {\ar"36"},
"29", {\ar"33"},
"29", {\ar"37"},
"35", {\ar"30"},
"30", {\ar"37"},
"31", {\ar"34"},
"31", {\ar"37"},
"32", {\ar"33"},
"32", {\ar"34"},
"33", {\ar"37"},
"35", {\ar"34"},
"36", {\ar"34"},
"34", {\ar"37"},
"36", {\ar"37"},
\end{xy}
\caption{The moment graph of $\mathcal{F}_4^a$. The vertices are labeled according to figure~\ref{Fig:TFixedA3}. The highlighted vertices correspond to the smooth $T$--fixed points.}\label{Fig:MomGrapgA3}
\end{figure}

To illustrate, let us describe in detail the graph around vertex $(22)$. There are $7$ edges connected to this vertex as depicted in figure~\ref{Fig:MomGraph22}. In particular this T--fixed point is not smooth.

\begin{figure}[h]
$$
\begin{array}{ccccccc}
20=\begin{array}{|c|}
\hline
\xymatrix@R=3pt@C=3pt"20"{&&*[F]{\cdot}\\&*{\cdot}\ar@{-}[r]&*{\cdot}\\*[F]{\cdot}\ar@{-}[r]&*[F]{\cdot}\ar@{-}[r]&*[F]{\cdot}\\*{\cdot}\ar@{-}[r]&*{\cdot}\ar@{-}[r]&*[F]{\cdot}\\*{\cdot}\ar@{-}[r]&*[F]{\cdot}&\\*{\cdot}&&}\\\hline
\end{array}
&
21=\begin{array}{|c|}
\hline
\xymatrix@R=3pt@C=3pt{&&*{\cdot}\\&*{\cdot}\ar@{-}[r]&*[F]{\cdot}\\*[F]{\cdot}\ar@{-}[r]&*[F]{\cdot}\ar@{-}[r]&*[F]{\cdot}\\*{\cdot}\ar@{-}[r]&*[F]{\cdot}\ar@{-}[r]&*[F]{\cdot}\\*{\cdot}\ar@{-}[r]&*{\cdot}&\\*{\cdot}&&}\\\hline
\end{array}
&
6=\begin{array}{|c|}
\hline
\xymatrix@R=3pt@C=3pt{&&*[F]{\cdot}\\&*{\cdot}\ar@{-}[r]&*[F]{\cdot}\\*[F]{\cdot}\ar@{-}[r]&*[F]{\cdot}\ar@{-}[r]&*[F]{\cdot}\\*{\cdot}\ar@{-}[r]&*{\cdot}\ar@{-}[r]&*{\cdot}\\*{\cdot}\ar@{-}[r]&*[F]{\cdot}&\\*{\cdot}&&}\\\hline
\end{array}
&
7=\begin{array}{|c|}
\hline
\xymatrix@R=3pt@C=3pt{&&*{\cdot}\\&*[F]{\cdot}\ar@{-}[r]&*[F]{\cdot}\\*[F]{\cdot}\ar@{-}[r]&*[F]{\cdot}\ar@{-}[r]&*[F]{\cdot}\\*{\cdot}\ar@{-}[r]&*{\cdot}\ar@{-}[r]&*[F]{\cdot}\\*{\cdot}\ar@{-}[r]&*{\cdot}&\\*{\cdot}&&}\\\hline
\end{array}\\&&&&\\(1,2)\downarrow&(4,5)\downarrow&(1,4)\downarrow&(2,5)\downarrow
\end{array}
$$
$$
\begin{array}{c}
22=\begin{array}{|c|}
\hline
\xymatrix@R=3pt@C=3pt{&&*{\cdot}\\&*{\cdot}\ar@{-}[r]&*[F]{\cdot}\\*[F]{\cdot}\ar@{-}[r]&*[F]{\cdot}\ar@{-}[r]&*[F]{\cdot}\\*{\cdot}\ar@{-}[r]&*{\cdot}\ar@{-}[r]&*[F]{\cdot}\\*{\cdot}\ar@{-}[r]&*[F]{\cdot}&\\*{\cdot}&&}\\\hline
\end{array}
\end{array}
$$
$$
\begin{array}{ccccccc}
&&&&&&\\&(3,4)\downarrow&&(3,5)\downarrow&&(3,6)\downarrow&\\
&
30=\begin{array}{|c|}
\hline
\xymatrix@R=3pt@C=3pt{&&*{\cdot}\\&*{\cdot}\ar@{-}[r]&*[F]{\cdot}\\*{\cdot}\ar@{-}[r]&*{\cdot}\ar@{-}[r]&*[F]{\cdot}\\*[F]{\cdot}\ar@{-}[r]&*[F]{\cdot}\ar@{-}[r]&*[F]{\cdot}\\*{\cdot}\ar@{-}[r]&*[F]{\cdot}&\\*{\cdot}&&}\\\hline
\end{array}
&
&
33=\begin{array}{|c|}
\hline
\xymatrix@R=3pt@C=3pt{&&*{\cdot}\\&*{\cdot}\ar@{-}[r]&*[F]{\cdot}\\*{\cdot}\ar@{-}[r]&*[F]{\cdot}\ar@{-}[r]&*[F]{\cdot}\\*{\cdot}\ar@{-}[r]&*{\cdot}\ar@{-}[r]&*[F]{\cdot}\\*[F]{\cdot}\ar@{-}[r]&*[F]{\cdot}&\\*{\cdot}&&}\\\hline
\end{array}
&&
35=\begin{array}{|c|}
\hline
\xymatrix@R=3pt@C=3pt{&&*{\cdot}\\&*{\cdot}\ar@{-}[r]&*[F]{\cdot}\\*{\cdot}\ar@{-}[r]&*[F]{\cdot}\ar@{-}[r]&*[F]{\cdot}\\*{\cdot}\ar@{-}[r]&*{\cdot}\ar@{-}[r]&*[F]{\cdot}\\*{\cdot}\ar@{-}[r]&*[F]{\cdot}&\\*[F]{\cdot}&&}\\\hline
\end{array}
\end{array}
$$
\caption{The moment graph around vertex $(22)$}
\label{Fig:MomGraph22}
\end{figure}
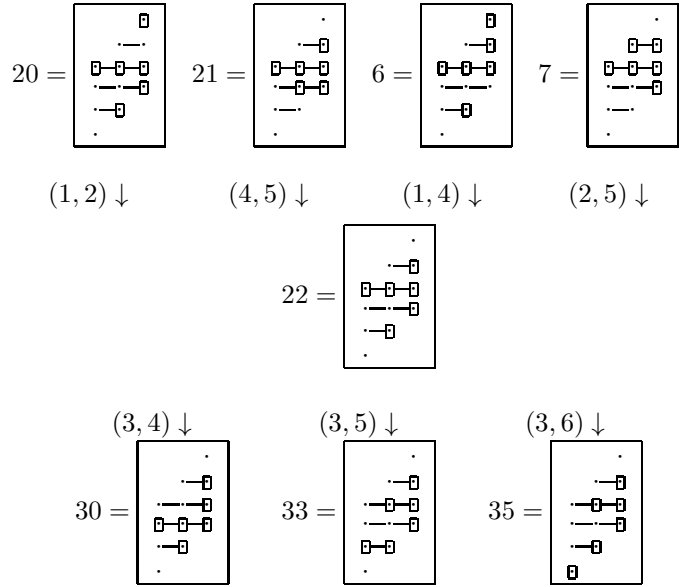
The arrow $(20)\rightarrow (22)$
corresponds to the following curve (in the basis \eqref{Eq:DiagBasis})
$$
(\langle v_1\rangle,\langle v_1, v_3\rangle,\langle v_3+\lambda v_2, v_1, v_4\rangle)\;\;\; \lambda\in\PP^1
$$
For $\lambda=0$ one gets the starting point $(20)$ of $\alpha$, for $\lambda=\infty$ one gets the end point $(22)$ of $\alpha$. Its direction around $(22)$  has negative $T_0$ weight while around $(20)$ it has positive weight. The corresponding (20)--effective pair is $(1, 2)$. The remaining labelings of figure~\ref{Fig:MomGraph22} are obtained similarly.
\newpage
\section*{Acknowledgements}
The work of Evgeny Feigin was partially supported
by the Russian President Grant MK-3312.2012.1, by the Dynasty Foundation,  
by the AG Laboratory HSE, RF government grant, ag. 11.G34.31.0023, and by the RFBR grants
12-01-00070 and 12-01-00944.
This study comprises research fundings from the `Representation Theory
in Geometry and in Mathematical Physics' carried out within The
National Research University Higher School of Economics' Academic Fund Program
in 2012, grant No 12-05-0014.
This study was carried out within the National Research University Higher School of Economics 
Academic Fund Program in 2012-2013, research grant No. 11-01-0017.

GCI thanks Francesco Esposito for helpful discussions.

\end{document}